\documentclass[a4paper,11pt]{article}

\usepackage[small,bf,hang]{caption} 

\usepackage{todonotes}
\usepackage[ruled,linesnumbered]{algorithm2e}
\usepackage{amsmath,amsthm,amssymb}
\usepackage{graphicx,subcaption,tikz,tkz-graph}
\usepackage{comment,xspace,paralist}							
\usepackage[shortlabels]{enumitem}	
\usepackage{hyperref}
\hypersetup{
    colorlinks=true,
    citecolor = blue,
    linkcolor=black,
    filecolor=magenta,
    urlcolor=cyan,
    bookmarks=true,
}
\usepackage[margin=1in]{geometry}

\newtheorem{theorem}{Theorem}
\newtheorem{lemma}{Lemma}
\newtheorem{claim}{Claim}

\newcommand{\case}[1]{\noindent {\bf Case #1.}}

\title{Colouring graphs with no induced six-vertex path or diamond}

\author{
Jan Goedgebeur\thanks{Department of Applied Mathematics, Computer Science and Statistics, Ghent University, 9000 Ghent, Belgium. Email: \texttt{jan.goedgebeur@ugent.be}.} \thanks{Department of Computer Science, KU Leuven Campus Kulak, 8500 Kortrijk, Belgium.} 
\and
Shenwei Huang\thanks{College of Computer Science, Nankai University, Tianjin 300350, China.  Email: \texttt{shenweihuang@nankai.edu.cn}. Supported by the National Natural Science Foundation of China (11801284) and Natural Science Foundation of Tianjin (20JCYBJC01190).}
\and
Yiao Ju\thanks{College of Computer Science, Nankai University, Tianjin 300350, China. Email: \texttt{545601319@qq.com}.}
\and
Owen Merkel\thanks{Department of Mathematics, Wilfrid Laurier University, Waterloo, Canada. Email: \texttt{owenmerkel@gmail.com}.}
}

\date{}

\begin{document}

\maketitle

\begin{abstract}

The diamond is the graph obtained by removing an edge from the complete graph on 4 vertices. A graph is ($P_6$, diamond)-free if it contains no induced subgraph isomorphic to a six-vertex path or a diamond. In this paper we show that the chromatic number of a ($P_6$, diamond)-free graph $G$ is no larger than the maximum of 6 and the clique number of $G$. We do this by reducing the problem to imperfect ($P_6$, diamond)-free graphs via the Strong Perfect Graph Theorem, dividing the imperfect graphs into several cases, and giving a proper colouring for each case. We also show that there is exactly one 6-vertex-critical ($P_6$, diamond, $K_6$)-free graph. Together with the  Lov{\'a}sz theta function, this gives a polynomial time algorithm to compute the chromatic number of ($P_6$, diamond)-free graphs.
\end{abstract}

\section{Introduction}

All graphs in this paper are finite and simple. For general graph theory notation we follow~\cite{BM08}. A \emph{$q$-colouring} of a graph $G$ assigns a colour from a colour set $\{1,\ldots,q\}$ to each vertex of $G$ such that adjacent vertices are assigned different colours. We say that a graph $G$ is \emph{$q$-colourable} if $G$ admits a $q$-colouring. The problem of deciding if a graph is \emph{$q$-colourable} is called the \emph{$q$-colouring problem}. This decision problem is NP-hard for general graphs for every $q \geq 3$, however there exist polynomial time algorithms when the input graphs are restricted to certain graph classes. For example, Chudnovsky, Spirkl, and Zhong~\cite{CSZ181,CSZ182} recently showed that the 4-colouring problem can be solved in polynomial time for the class of $P_6$-free graphs.
The \emph{chromatic number} of a graph $G$, denoted by $\chi(G)$, is the minimum number $q$ for which $G$ is $q$-colourable. 

Let $P_n$, $C_n$ and $K_n$ denote the path, cycle and complete graph on $n$ vertices, respectively. The \emph{diamond} is the graph obtained from $K_4$ by removing an edge. For two graphs $G$ and $H$, we use $G+H$ to denote the \emph{disjoint union} of $G$ and $H$, and $G\vee H$ 
to denote the graph obtained from the disjoint union of $G$ and $H$ by adding an edge between every vertex in $G$ and every vertex in $H$. For a positive integer $r$, we use $rG$ to denote the disjoint union of $r$ copies of $G$. A \emph{hole} is an induced cycle on 4 or more vertices. An \emph{antihole} is the complement of a hole. A hole or antihole is \emph{odd} or \emph{even} if it has an odd or even number of vertices. We say that a graph $G$ \emph{contains} a graph $H$ if an induced subgraph of $G$ is isomorphic to $H$. A graph $G$ is \emph{$H$-free} if it does not contain $H$. For a family $\mathcal{H}$ of graphs, $G$ is \emph{$\mathcal{H}$-free} if $G$ is $H$-free for every $H\in\mathcal{H}$. The graphs in $\mathcal{H}$ are the \emph{forbidden induced subgraphs} of the family of $\mathcal{H}$-free graphs. We write $(H_1,\ldots,H_n)$-free instead of $\{H_1,\ldots,H_n\}$-free. A \emph{clique} is a vertex set whose elements are pairwise adjacent. A vertex set is \emph{stable} if its elements are pairwise nonadjacent, and \emph{nonstable} otherwise. The \emph{clique number} of $G$, denoted by $\omega(G)$, is the size of a largest clique in $G$. Obviously, $\chi(G)\geq\omega(G)$ for any graph $G$.

A graph family $\mathcal{G}$ is \emph{hereditary} if $G\in\mathcal{G}$ implies that every induced subgraph of $G$ belongs to $\mathcal{G}$. Obviously, $\mathcal{G}$ is hereditary if and only if $\mathcal{G}$ is the class of $\mathcal{H}$-free graphs for some $\mathcal{H}$. A graph $G$ is \emph{perfect} if $\chi(H)=\omega(H)$ for each induced subgraph $H$ of $G$, and \emph{imperfect} otherwise. As a generalisation of perfect graphs, Gy{\'a}rf{\'a}s~\cite{Gy87} introduced the $\chi$-bounded graph families. A hereditary graph family $\mathcal{G}$ is \emph{$\chi$-bounded} if there is a function $f$ such that $\chi(G)\leq f(\omega(G))$ for every $G\in\mathcal{G}$. The function $f$ is called a \emph{$\chi$-binding function}. 
Chudnovsky, Robertson, Seymour and Thomas~\cite{CRST06} characterized the family of perfect graphs by forbidden induced subgraphs:

\begin{theorem}[\cite{CRST06}]\label{thm:SPGT}
A graph is perfect if and only if it does not contain an odd hole or an odd antihole as an induced subgraph.
\end{theorem}

In other words, the family of (odd hole, odd antihole)-free graphs is $\chi$-bounded, and its $\chi$-binding function is the identity function. Based on this theorem, researchers studied various graph families with two forbidden induced subgraphs and found several families that have a linear $\chi$-binding function. Note that every graph family mentioned in this paragraph forbids odd holes and odd antiholes on 7 or more vertices. It then follows from Theorem~\ref{thm:SPGT} that every graph in these graph families is either perfect or contains a $C_5$. In particular, Gaspers and Huang~\cite{GH17} showed that $\chi(G)\leq\frac{3}{2}\omega(G)$ for every ($P_6$,$C_4$)-free graph $G$. Karthick and Maffray~\cite{KMa18} improved the $\chi$-binding function of ($P_6$,$C_4$)-free graphs to $\frac{5}{4}\omega(G)$. The graph $P_4\vee K_1$ is called a \emph{gem}, and a \emph{co-gem} is the complement of a gem. Cameron, Huang and Merkel~\cite{CHM19} showed that $\chi(G)\leq\lfloor\frac{3}{2}\omega(G)\rfloor$ for every ($P_5$, gem)-free graph $G$. Chudnovsky, Karthick, Maceli, and Maffray~\cite{MTPF18} improved the $\chi$-binding function of ($P_5$, gem)-free graphs to $\lceil\frac{5}{4}\omega(G)\rceil$. Karthick and Maffray~\cite{KMa182} showed that $\chi(G)\leq\lceil\frac{5}{4}\omega(G)\rceil$ for every (gem, co-gem)-free graph $G$, and in~\cite{KM16} they showed that every ($P_5$, diamond)-free graph $G$ satisfies $\chi(G)\leq\omega(G)+1$. For the family of ($P_6$, diamond)-free graphs, Karthick and Mishra~\cite{KM18} showed that every ($P_6$, diamond)-free graph $G$ satisfies $\chi(G)\leq2\omega(G)+5$. In the same paper, they proved that every ($P_6$, diamond, $K_4$)-free graph is 6-colourable. Finally, Cameron, Huang, and Merkel~\cite{CHM18} improved the $\chi$-binding function of ($P_6$, diamond)-free graphs to $\omega(G)+3$.

\subsection*{Our Contributions}
%
In this paper, we prove that every ($P_6$, diamond)-free graph $G$ satisfies $\chi(G)\leq \max\{6,\omega(G)\}$ (cf.~\autoref{thm:bound} in \autoref{sec:bound}).  We do this by reducing the problem to imperfect ($P_6$, diamond)-free graphs via the Strong Perfect Graph Theorem, dividing the imperfect graphs into several cases, and giving a proper colouring for each case.

Furthermore, we prove that the chromatic number of ($P_6$, diamond)-free graphs can be determined in polynomial time (cf.~\autoref{thm:poly} in \autoref{sec:poly}). In particular, we show that there is exactly one 6-vertex-critical ($P_6$, diamond, $K_6$)-free graph. Together with the  Lov{\'a}sz theta function, this gives a polynomial time algorithm to compute the chromatic number of ($P_6$, diamond)-free graphs.

Our results are an improvement of the result of Cameron, Huang and Merkel~\cite{CHM18}, and answer an open question from~\cite{CHM18}. We believe that our proof technique for polynomial time solvability may also be useful for other graph families (see \autoref{sec:conclude}).

The remainder of the paper is organised as follows. We present some preliminaries in \autoref{sec:pre}
and show some structural properties of imperfect ($P_6$, diamond)-free graphs in \autoref{sec:imperfect}.
We prove the $\chi$-bound in \autoref{sec:bound} and prove 
that the chromatic number can be determined in polynomial time in \autoref{sec:poly}. We end with some open problems in \autoref{sec:conclude}.

\section{Preliminaries}\label{sec:pre}

Let $G=(V,E)$ be a graph. A \emph{neighbour} of a vertex $v$ is a vertex adjacent to $v$. The \emph{neighbourhood} of a vertex $v$, denoted by $N_G(v)$, is the set of neighbours of $v$. The \emph{degree} of a vertex $v$, denoted by $d(v)$, is the number of neighbours of $v$. We denote the minimum degree of the vertices of $G$ by $\delta(G)$. For $X \subseteq V$, let $N_G(X)=\bigcup_{v\in X}N_G(v)\setminus X$. For $x\in V$ (or $X\subseteq V$) and $S\subseteq V$, let $N_S(x)=N_G(x)\cap S$ (or $N_S(X)=N_G(X)\cap S$). For $x\in V$ (or $X\subseteq V$) and $Y\subseteq V$, we say that $x$ (or $X$) is \emph{complete} (resp.~\emph{anti-complete}) to $Y$ if $x$ (or every vertex in $X$) is adjacent (resp.~nonadjacent) to every vertex in $Y$. We denote the complement of $G$ by $\overline{G}$. For $S\subseteq V$, let $G[S]$ denote the subgraph of $G$ induced by $S$. We often write $S$ for $G[S]$ if the context is clear. We say that $S$ induces an $H$ if $G[S]$ is isomorphic to $H$. A clique $K\subseteq V$ is a \emph{clique cutset} if $G-K$ has more components than $G$. Two vertices $u,v\in V$ are \emph{comparable} if they are nonadjacent, and either $N_G(u)\subseteq N_G(v)$ or $N_G(v)\subseteq N_G(u)$. A component of a graph is \emph{trivial} if it has only one vertex, and \emph{nontrivial} otherwise. We say that the edges between two vertex sets $X$ and $Y$ form a \emph{matching} if every vertex in $X$ has at most one neighbour in $Y$, and every vertex in $Y$ has at most one neighbour in $X$.

Gr{\" o}tschel, Lov{\' a}sz and Schrijver~\cite{GLS84} showed that the chromatic number of a perfect graph can be computed in polynomial time. In that paper, the authors used the Lov{\'a}sz theta function:
\begin{center}
$\vartheta(G):=max\{\sum\limits_{i,j=1}^n b_{ij}:$ \\
$B=(b_{ij})$ is positive semidefinite with trace at most 1, and $b_{ij}=0$ if $ij\in E\}$.\\
\end{center}
The Lov{\'a}sz theta function satisfies that $\omega(G)\leq\vartheta(\overline{G})\leq\chi(G)$ for any graph $G$, and can be calculated in polynomial time~\cite{GLS84}.

A graph $G$ is \emph{$k$-vertex-critical} if $\chi(G)=k$, and every proper induced subgraph of $G$ has chromatic number smaller than $k$. The following properties of $k$-vertex-critical graphs are well-known.

\begin{lemma}[\cite{BM08}]\label{lem:critical}
If $G$ is a $k$-vertex-critical graph, then $G$ is connected, has no clique cutsets or comparable vertices, and $\delta(G) \geq k-1$.
\end{lemma}

It is easy to see that $G$ is not ($k-1$)-colourable if and only if $G$ contains a $k$-vertex-critical graph. This simple observation has an important algorithmic implication:

\begin{theorem}[Folklore]\label{thm:exhaustive}
If a hereditary graph family $\mathcal{G}$ has a finite number of $k$-vertex-critical graphs, then the ($k-1$)-colouring problem can be solved in polynomial time for $\mathcal{G}$ by simply testing if the input graph contains any of these $k$-vertex-critical graphs as induced subgraph.
\end{theorem}


The concept of $k$-vertex-critical graphs is also important in the context of certifying algorithms~\cite{MMNS11}. 
An algorithm is \textit{certifying} if, along with the answer given by the algorithm, it also gives a certificate which allows to verify in polynomial time that the output of the algorithm is indeed correct. In case of the $k$-colouring problem, a canonical certificate for yes-instances would be a proper $k$-colouring of the graph while a canonical certificate for no-instances would be a $(k+1)$-vertex-critical graph.

We will also use the following two theorems in our proof:

\begin{theorem}[\cite{CSZ181,CSZ182}]\label{thm:4CP6}
The 4-colouring problem can be solved in polynomial time for the class of $P_6$-free graphs.
\end{theorem}

\begin{theorem}[\cite{RST02}]\label{thm:P6C3}
Let $G$ be a $(P_6,K_3)$-free graph with no comparable vertices. Then $G$ is 4-colourable. Furthermore, $G$ is not 3-colourable if and only if it contains the Gr{\"o}tzsch graph as an induced subgraph and is an induced subgraph of the 16-vertex Clebsch graph. (See \autoref{fig:Clebsch} for drawings of the Clebsch graph and the Gr{\"o}tzsch graph.)
\end{theorem}

\begin{figure}[h!]
\centering
\begin{tikzpicture}[scale=1]
\tikzstyle{vertex}=[circle, draw, fill=white, inner sep=1pt, minimum size=5pt]
    \node[vertex](1) at (0,0) {0};
    \node[vertex](2) at (0,2) {1};
    \node[vertex](3) at (-{2*sin(36)},-{2*cos(36}) {2};
    \node[vertex](4) at ({2*cos(18)},{2*sin(18}) {3};
    \node[vertex](5) at (-{2*cos(18)},{2*sin(18}) {4};
    \node[vertex](6) at ({2*sin(36)},-{2*cos(36}) {5};
    \node[vertex](7) at (0,-2) {6};
    \node[vertex](8) at ({2*sin(36)},{2*cos(36}) {7};
    \node[vertex](9) at (-{2*cos(18)},-{2*sin(18}) {8};
    \node[vertex](10) at ({2*cos(18)},-{2*sin(18}) {9};
    \node[vertex](11) at (-{2*sin(36)},{2*cos(36}) {10};
    \node[vertex](12) at (0,-3) {11};
    \node[vertex](13) at ({3*sin(36)},{3*cos(36}) {12};
    \node[vertex](14) at (-{3*cos(18)},-{3*sin(18}) {13};
    \node[vertex](15) at ({3*cos(18)},-{3*sin(18}) {14};
    \node[vertex](16) at (-{3*sin(36)},{3*cos(36}) {15};

    \Edge(1)(2)
 	\Edge(1)(3)
 	\Edge(1)(4)
 	\Edge(1)(5)
	\Edge(1)(6)
	\Edge(2)(8)
	\Edge(2)(11)
	\Edge(3)(7)
	\Edge(3)(9)
	\Edge(4)(8)
	\Edge(4)(10)
	\Edge(5)(9)
	\Edge(5)(11)
	\Edge(6)(7)
	\Edge(6)(10)
	\Edge(7)(8)
	\Edge(8)(9)
	\Edge(9)(10)
	\Edge(10)(11)
	\Edge(7)(11)
    \Edge(4)(12)
    \Edge(5)(12)
    \Edge(7)(12)
    \Edge(5)(13)
    \Edge(6)(13)
    \Edge(8)(13)
    \Edge(2)(14)
    \Edge(6)(14)
    \Edge(9)(14)
    \Edge(2)(15)
    \Edge(3)(15)
    \Edge(10)(15)
    \Edge(3)(16)
    \Edge(4)(16)
    \Edge(11)(16)
    \Edge(12)(14)
    \Edge(12)(15)
    \Edge(13)(15)
    \Edge(13)(16)
    \Edge(14)(16)

\end{tikzpicture}
\caption{The Clebsch graph. The 11 vertices with indices $0,\ldots,10$ induce the Gr{\"o}tzsch graph.}\label{fig:Clebsch}
\end{figure}
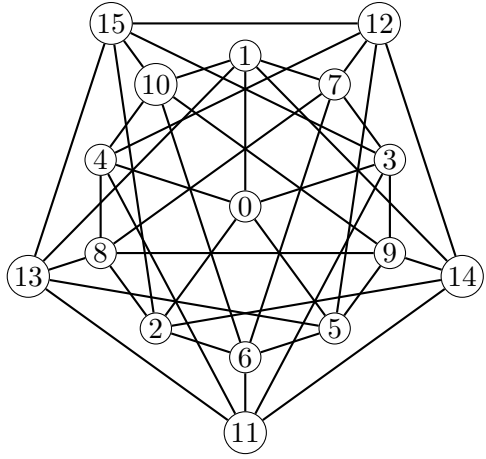

\section{The Structure of Imperfect ($P_6$, diamond)-Free Graphs}\label{sec:imperfect}

In this section we study the structure of imperfect ($P_6$, diamond)-free graphs. By \autoref{thm:SPGT}, every imperfect ($P_6$, diamond)-free graph contains a $C_5$. Let $G=(V,E)$ be an imperfect ($P_6$, diamond)-free graph. We follow the notation of~\cite{CHM18} and partition $V$ into the following subsets:

Let $Q$=\{$v_1,v_2,v_3,v_4,v_5$\} induce a $C_5$ in $G$ with edges $v_iv_{i+1}$ for $i$=1,$\ldots$,5, with all indices modulo 5.

$A_i$=\{$v\in V\backslash Q:N_Q(v)=\{v_i\}$\},

$B_{i,i+1}$=\{$v\in V\backslash Q:N_Q(v)=\{v_i,v_{i+1}\}$\},

$C_{i,i+2}$=\{$v\in V\backslash Q:N_Q(v)=\{v_i,v_{i+2}\}$\},

$F_i$=\{$v\in V\backslash Q:N_Q(v)=\{v_i,v_{i-2},v_{i+2}\}$\},

$Z$=\{$v\in V\backslash Q:N_Q(v)=\emptyset$\}.

Let $A$=$\bigcup_{i=1}^5A_i$, $B$=$\bigcup_{i=1}^5B_{i,i+1}$, $C$=$\bigcup_{i=1}^5C_{i,i+2}$, $F$=$\bigcup_{i=1}^5F_i$. Since $G$ is diamond-free, any vertex in $V\backslash Q$ cannot be adjacent to three sequential vertices $v_i,v_{i+1},v_{i+2}$ in $Q$, then $V$=$Q\cup A\cup B\cup C\cup F\cup Z$.

In~\cite{CHM18} the following 21 properties of these subsets are proved:

\begin{enumerate}[(1)]
\setlength{\itemsep}{0pt}

\item Each component of $A_i$ is a clique.\label{item:1}

\item The sets $A_i$ and $A_{i+1}$ are anti-complete.\label{item:2}

\item The sets $A_i$ and $A_{i+2}$ are complete.\label{item:3}

\item Each $B_{i,i+1}$ is a clique.\label{item:4}

\item The set $B$=$B_{i,i+1}\cup B_{i+2,i+3}$ for some $i$.\label{item:5}

\item The set $B_{i,i+1}$ is anti-complete to $A_i\cup A_{i+1}$.\label{item:6}

\item The set $B_{i,i+1}$ is complete to $A_{i-1}\cup A_{i+2}$.\label{item:7}

\item Each $C_{i,i+2}$ is a stable set.\label{item:8}

\item Each vertex in $C_{i,i+2}$ is either complete or anti-complete to each component of $A_i$ and $A_{i+2}$.\label{item:9}

\item Each vertex in $C_{i,i+2}$ has at most one neighbour in each component of $A_{i+1}$, $A_{i+3}$ and $A_{i+4}$.\label{item:10}

\item Each vertex in $C_{i,i+2}$ is anti-complete to each nontrivial component of $A_{i+1}$.\label{item:11}

\item The set $C_{i,i+2}$ is anti-complete to $B_{j,j+1}$ if $j\neq i+3$. Moreover each vertex in $C_{i,i+2}$ has at most one neighbour in $B_{i+3,i+4}$.\label{item:12}

\item Each $F_i$ has at most one vertex. Moreover, $F$ is a stable set.\label{item:13}

\item The set $F_i$ is anti-complete to $A_{i+2}\cup A_{i+3}$.\label{item:14}

\item Each vertex in $F_i$ is either complete or anti-complete to each component of $A_i$.\label{item:15}

\item Each vertex in $F_i$ has at most one neighbour in each component of $A_{i+1}$ and $A_{i+4}$.\label{item:16}

\item The set $F_i$ is anti-complete to $B_{j,j+1}$ if $j\neq i+2$ and complete to $B_{j,j+1}$ if $j=i+2$.\label{item:17}

\item The set $F_i$ is anti-complete to $C_{j,j+2}$ if $j\neq i-1$.\label{item:18}

\item If $A_i$ is not stable, then $A_{i+2}=A_{i+3}=B_{i+1,i+2}=B_{i-1,i-2}=\emptyset$.\label{item:19}

\item If $A_i$ is not empty, then each of $B_{i+1,i+2}$ and $B_{i-1,i-2}$ contains at most one vertex.\label{item:20}

\item The set $Z$ is anti-complete to $A\cup B$.\label{item:21}

\noindent Now we prove some new properties which we will use in our proofs in \autoref{sec:bound} and \autoref{sec:poly}:


\item If $A_i$, $B_{i+1,i+2}$ and $B_{i+3,i+4}$ are all nonempty, then $C_{i+1,i+3}$ (resp.\ $C_{i+2,i+4}$) is anti-complete to $A_{i+2}$ (resp.\ $A_{i+3}$).\label{item:22}

\vspace{-1em}

\begin{proof}
Suppose that $a_1\in A_i$, $a_2\in A_{i+2}$, $b_1\in B_{i+1,i+2}$, $b_2\in B_{i+3,i+4}$, $c\in C_{i+1,i+3}$ such that $a_2c\in E$. By~\ref{item:3}, $a_1a_2\in E$. By~\ref{item:6}, $b_1a_2\notin E$. By~\ref{item:7}, $a_1b_1,a_1b_2,a_2b_2\in E$. By~\ref{item:12}, $b_1c,b_2c\notin E$. Then either \{$a_1,a_2,b_2,c$\} induces a diamond or \{$b_1,a_1,a_2,c,v_{i+3},v_{i+4}$\} induces a $P_6$, depending on whether $a_1$ and $c$ are adjacent.
\end{proof}

\item Suppose that $B_{i,i+1}$ is nonempty. Then $C_{i,i+2}\cup C_{i-2,i}$ (resp.\ $C_{i-1,i+1}\cup C_{i+1,i+3}$) is complete to $A_{i-1}$ (resp.\ $A_{i+2}$). Moreover if $A_{i-1}$ (resp.\ $A_{i+2}$) is nonempty, then $C_{i,i+2}$ and $C_{i-2,i}$ (resp.\ $C_{i-1,i+1}$ and $C_{i+1,i+3}$) are anti-complete.\label{item:23}

\vspace{-1em}

\begin{proof}
Let $a\in A_{i-1}$, $b\in B_{i,i+1}$. Suppose there is a vertex $c_1\in C_{i,i+2}$ such that $ac_1\notin E$ or $c_2\in C_{i-2,i}$ such that $ac_2\notin E$. By~\ref{item:12}, $bc_1,bc_2\notin E$. Then \{$c_1,v_{i+2},v_{i-2},v_{i-1},a,b$\} or \{$c_2,v_{i-2},v_{i-1},a,b,v_{i+1}$\} induces a $P_6$. This proves the first part of the claim. Suppose that $c_1\in C_{i,i+2}$ and $c_2\in C_{i-2,i}$ are adjacent, then \{$a,c_1,c_2,v_i$\} induces a diamond.
\end{proof}

\item Suppose that $A_i\cup A_{i+1}$ is nonempty. Then $B_{i,i+1}$ is complete to $A_{i+3}$. Moreover if $B_{i,i+1}$ contains two or more vertices, then $A_{i+3}$ is empty.\label{item:24}
\vspace{-1em}
\begin{proof}
By symmetry suppose that $a_1\in A_i$, $a_2\in A_{i+3}$, $b\in B_{i,i+1}$ such that $ba_2\notin E$. By~\ref{item:3}, $a_1a_2\in E$. By~\ref{item:6}, $ba_1\notin E$. Then \{$b,v_i,a_1,a_2,v_{i+3},v_{i+2}$\} induces a $P_6$. This proves the first part of the claim. Suppose that $b_1,b_2\in B_{i,i+1}$, $a\in A_{i+3}$. By~\ref{item:4}, $b_1b_2\in E$. Then \{$v_i,b_1,b_2,a$\} induces a diamond.
\end{proof}

\item If $B_{i,i+1}$ is nonempty, then $C_{i-2,i}$ (resp.\ $C_{i+1,i+3}$) is anti-complete to $A_i$ (resp.\ $A_{i+1}$).\label{item:25}
\vspace{-1em}
\begin{proof}
Suppose that $a\in A_i$, $b\in B_{i,i+1}$, $c\in C_{i-2,i}$ such that $ac\in E$. By~\ref{item:6}, $ab\notin E$. By~\ref{item:12}, $bc\notin E$. Then \{$a,c,v_{i-2},v_{i+2},v_{i+1},b$\} induces a $P_6$.
\end{proof}

\item If $B_{i,i+1}$ is nonempty, then $Z$ is anti-complete to $C_{i+1,i+3}\cup C_{i+3,i}$.\label{item:26}
\vspace{-1em}
\begin{proof}
By symmetry suppose that $b\in B_{i,i+1}$, $c\in C_{i+1,i+3}$ and $z\in Z$ such that $cz\in E$. By~\ref{item:12}, $bc\notin E$. By~\ref{item:21}, $bz\notin E$. Then $\{z,c,v_{i+3},v_{i+4},v_i,b\}$ induces a $P_6$.
\end{proof}
\end{enumerate}

\section{The $\chi$-Bound of ($P_6$, diamond)-Free Graphs}\label{sec:bound}

In this section, we prove the following theorem.

\begin{theorem}\label{thm:bound}
Let $G$ be a ($P_6$, diamond)-free graph, then $\chi(G)\leq max\{6,\omega(G)$\}.
\end{theorem}

This theorem is an improvement of the result of Cameron, Huang and Merkel~\cite{CHM18} that $\chi(G) \leq \omega(G)+3$ for a ($P_6$, diamond)-free graph $G$. To prove \autoref{thm:bound}, we use \autoref{thm:imperfect} below.

\begin{theorem}\label{thm:imperfect}
Let $G$ be a connected imperfect ($P_6$, diamond)-free graph with no clique cutsets or comparable vertices. Then $\chi(G)\leq max\{6,\omega(G)$\}.
\end{theorem}

\begin{proof}[Proof of \autoref{thm:bound}]
If $G$ is perfect, then $\chi(G)=\omega(G)$. If $G$ is disconnected and is the union of components \{$H_1,\ldots,H_n$\}, then $\chi(G)$=$max\{\chi(H_1),\ldots,\chi(H_n)\}$ and $\omega(G)$=$max\{\omega(H_1),\ldots,\omega(H_n)\}$. If $G$ is connected and contains a clique cutset $S$, $G-S$ is the disjoint union of subgraphs \{$H_1,\ldots,H_n$\}, then $\chi(G)$=max\{$\chi(G[V(H_1)\cup S],\ldots,\chi(G[V(H_n)\cup S]$\} and $\omega(G)$=max\{$\omega(G[V(H_1)\cup S],\ldots,\omega(G[V(H_n)\cup S]$\}. If $u$ and $v$ are nonadjacent vertices in $G$ such that $N(u)\subseteq N(v)$, then $\chi(G)$=$\chi(G-u)$ and $\omega(G)$=$\omega(G-u)$. Then the theorem follows from \autoref{thm:imperfect}.
\end{proof}

\begin{proof}[Proof of \autoref{thm:imperfect}]
Let $G=(V,E)$. We use the partitioning of $V$ from \autoref{sec:imperfect}. Let $\omega:=max\{6,\omega(G)\}$. In the following we give a $\omega$-colouring for $G$. When we say that two vertex sets \textit{do not conflict}, we mean that no edges between the two sets have two ends of the same colour.

\begin{claim}\label{clm:Z colouring}
If $V\setminus Z$ has a $\omega$-colouring such that one colour of $\{1,\ldots,\omega\}$ is not used in $C\cup F$, and $F$ is coloured with at most 2 colours, then we can extend it to a $\omega$-colouring of $G$.
\end{claim}

\begin{proof}
Let $K$ be an arbitrary component of $Z$. By~\ref{item:21}, $N_G(K)\subseteq C\cup F$. If $K$ contains exactly one vertex, then colour it with the colour not used in $C\cup F$. Now suppose that $K$ contains at least two vertices.

It is proved in~\cite{CHM18} that if $K$ has no neighbours in $C$, then $K$ is 3-colourable. (See Claim 1 in~\cite{CHM18}. The conditions that $G$ is connected and has no clique cutsets or comparable vertices are necessary.) Since $F$ is coloured with at most 2 colours, we can colour $K$ with the $\omega-2$ colours not used in $F$ if $K$ has no neighbour in $C$. Now suppose that $K$ has a neighbour in $C$.

Let $c\in C_{i,i+2}$ be a neighbour of some vertex $z_1\in K$. If there is a vertex $z_2$ in $K$ such that $z_1z_2\in E$ and $z_2c\notin E$, then $\{z_2,z_1,c,v_i,v_{i+4},v_{i+3}\}$ induces a $P_6$. So $c$ is complete to $K$. Then $\{c\}\cup K$ is a clique since $G$ is diamond-free. If there is more than one vertex in $C$ complete to $K$, then $K\cup N_C(K)$ is a clique  since $K$ contains a $K_2$.

Suppose that $c_1,c_2,c_3$ are three neighbours of $K$ in $C$, then $\{c_1,c_2,c_3\}$ is a clique. By~\ref{item:8}, we may assume that $c_1\in C_{i,i+2}$, $c_2\in C_{i+1,i+3}$, $c_3\in C_{i+2,i+4}\cup C_{i+3,i}$. Then either $\{c_1,c_2,c_3,v_{i+2}\}$ induces a diamond if $c_3\in C_{i+2,i+4}$, or $\{c_1,c_2,c_3,v_i\}$ induces a diamond if $c_3\in C_{i+3,i}$. So $K$ has at most two neighbours in $C$.

Since $K$ is a clique of size at least two, each vertex in $N_F(K)$ is either complete to $K$ or adjacent to exactly one vertex in $K$, and no more than one vertex in $F$ is complete to $K$ by~\ref{item:13}. If $f\in F$ is complete to $K$, then $K\cup N_C(K)\cup\{f\}$ is a clique, and then $N_C(K)$ contains only one vertex by~\ref{item:8} and~\ref{item:18}. Let $Y=\{v\in C\cup F: v$ is complete to $K\}$, then $Y$ is a clique of size at most~2. Then each vertex $z$ in $K$ must have a neighbour in $F\setminus Y$, or else $Y\cup(K\setminus \{z\})$ is a clique cutset.

Now suppose that $K$ contains three vertices $z_1,z_2,z_3$. Let $f_1,f_2,f_3\in F$ such that $z_if_j\in E$ if and only if $i=j$ for $i,j\in\{1,2,3\}$. By symmetry we may assume that $f_1\in F_i$, $f_2\in F_{i+1}$, $f_3\in F_{i+2}\cup F_{i+3}$. Then either $\{f_1,v_{i+3},f_2,z_2,z_3,f_3\}$ induces a $P_6$ if $f_3\in F_{i+2}$, or $\{f_1,v_i,f_3,z_3,z_2,f_2\}$ induces a $P_6$ if $f_3\in F_{i+3}$. So $K$ contains exactly two vertices. Since $K$ has at most 2 neighbours in $C$, and $F$ is coloured with at most 2 colours, there are at least 2 colours left for $K$ since $\omega\geq6$.

Colour each component of $Z$ this way, then $Z$ and $V\setminus Z$ do not conflict.
\end{proof}

Now we will consider all cases and give an $\omega$-colouring for each case. In the following, when we say that we colour a set with a certain colour, we mean that we colour each vertex in the set with that colour. When we say that we colour a clique or the disjoint union of some cliques with a set of colours, we mean that we colour each clique using the smallest colours available unless stated otherwise. For example, let $A_1$ be the disjoint union of three cliques $K_1,K_2,K_3$ of size 1,2,3 in order, when we say that we colour $A_1$ with colours in $1,3,4,5$, we mean that we colour $K_1$ with colour $1$, $K_2$ with colours $1,3$ and $K_3$ with colours $1,3,4$.

By~\ref{item:5}, we may assume that $B=B_{2,3}\cup B_{4,5}$. In the following, we give a proper $\omega$-colouring of $V\setminus Z$ such that one colour of $\{1,\ldots,\omega\}$ is not used in $C\cup F$, and $F$ is coloured with at most 2 colours. Then by \autoref{clm:Z colouring}, $G$ is $\omega$-colourable. We consider three cases: both $B_{2,3}$ and $B_{4,5}$ are nonempty, exactly one of $B_{2,3}$ and $B_{4,5}$ is nonempty, and $B$ is empty.

\case{1} Both $B_{2,3}$ and $B_{4,5}$ are nonempty.

\case{1.1} Both $B_{2,3}$ and $B_{4,5}$ contain at least two vertices.

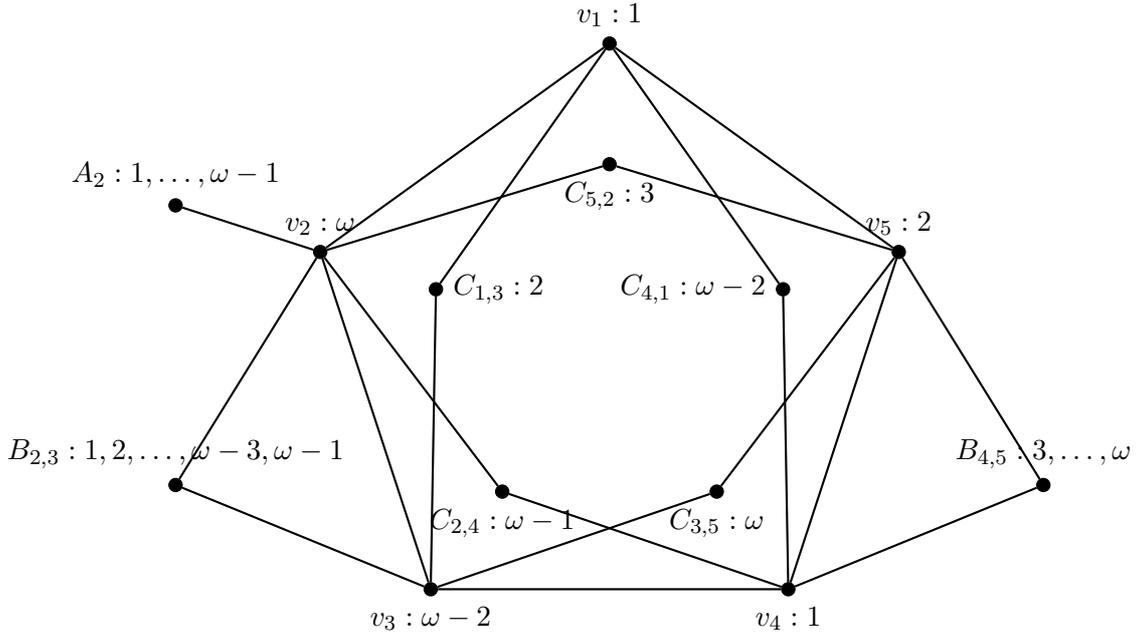
\begin{figure}[h!]
\centering
\begin{tikzpicture}[scale=2]
\tikzstyle{vertex}=[circle, draw, fill=black, inner sep=1pt, minimum size=5pt]
    \node[vertex, label=${v_1:1}$](1) at (0,2) {};
    \node[vertex, label=${v_2:\omega}$](2) at (-{2*cos(18)},{2*sin(18}) {};
    \node[vertex, label=below:${v_3:\omega-2}$](3) at (-{2*sin(36)},-{2*cos(36}) {};
    \node[vertex, label=below:${v_4:1}$](4) at ({2*sin(36)},-{2*cos(36}) {};
    \node[vertex, label=${v_5:2}$](5) at ({2*cos(18)},{2*sin(18}) {};
    \node[vertex, label=${B_{2,3}:1,2,\ldots,\omega-3,\omega-1}$](6) at (-{3*cos(18)},-{3*sin(18}) {};
    \node[vertex, label=${B_{4,5}:3,\ldots,\omega}$](7) at ({3*cos(18)},-{3*sin(18}) {};
    \node[vertex, label=${A_2:1,\ldots,\omega-1}$](8) at (-{3*cos(18)},{3*sin(18}) {};
    \node[vertex, label=below:${C_{5,2}:3}$](9) at (0,1.2) {};
    \node[vertex, label=right:${C_{1,3}:2}$](10) at (-{1.2*cos(18)},{1.2*sin(18}) {};
    \node[vertex, label=below:${C_{2,4}:\omega-1}$](11) at (-{1.2*sin(36)},-{1.2*cos(36}) {};
    \node[vertex, label=below:${C_{3,5}:\omega}$](12) at ({1.2*sin(36)},-{1.2*cos(36}) {};
    \node[vertex, label=left:${C_{4,1}:\omega-2}$](13) at ({1.2*cos(18)},{1.2*sin(18}) {};

    \Edge(1)(2)
    \Edge(2)(3)
    \Edge(3)(4)
    \Edge(4)(5)
    \Edge(1)(5)
    \Edge(2)(6)
    \Edge(3)(6)
    \Edge(4)(7)
    \Edge(5)(7)
    \Edge(2)(8)
    \Edge(2)(9)
    \Edge(5)(9)
    \Edge(1)(10)
    \Edge(3)(10)
    \Edge(2)(11)
    \Edge(4)(11)
    \Edge(3)(12)
    \Edge(5)(12)
    \Edge(1)(13)
    \Edge(4)(13)

\end{tikzpicture}
\caption{Diagram of the colouring of $Q\cup A\cup B\cup C$ in the proof of Case 1.1.} Edges in $A\cup B\cup C$ are not drawn.\label{fig:diag}
\end{figure}

By~\ref{item:20}, $A_1\cup A_3\cup A_4$ is empty. By~\ref{item:24}, one of $A_2$ and $A_5$ is empty. By symmetry we assume that $A_5$ is empty. In the following we give a proper colouring of $Q$,$B$,$A$,$C$,$F$ in turn. When we give a colouring of a vertex set, we prove that it does not conflict with the vertices we already coloured. So the whole colouring is proper. We give a diagram of the colouring of $Q\cup A\cup B\cup C$ in \autoref{fig:diag}.

$\bullet$ Colour $v_1,\ldots,v_5$ with colours $1,\omega,\omega-2,1,2$ in order.

$\bullet$ By~\ref{item:4}, each of $B_{2,3}$ and $B_{4,5}$ is a clique of size at most $\omega-2$. Colour $B_{4,5}$ with colours in \{$3,\ldots,\omega$\}. Exceptionally, when $B_{4,5}$ contains exactly $\omega-3$ vertices, we use colours $\{3,4,\ldots,\omega-2,\omega\}$ instead of colours $\{3,\ldots,\omega-1\}$.

$\bullet$ Since $G$ is diamond-free, the edges between $B_{2,3}$ and $B_{4,5}$ form a matching. Sort the vertices in $B_{2,3}$ by the colour of their neighbours in $B_{4,5}$ in ascending order, putting the vertices with no neighbours in $B_{4,5}$ at the end of the queue. Then colour the vertices in the queue with colours $1,2,\ldots,\omega-3,\omega-1$ in order. This ensures that for each edge between the two sets, its end in $B_{2,3}$ has a smaller colour than the end in $B_{4,5}$. So $B_{2,3}$ and $B_{4,5}$ do not conflict.

$\bullet$ Let $K$ be an arbitrary component of $A_2$. By~\ref{item:1}, $K$ is a clique of size at most $\omega-1$. Then the edges between $B_{4,5}$ and $K$ form a matching since $G$ is diamond-free. Sort the vertices in $K$ by the colour of their neighbours in $B_{4,5}$ in ascending order, putting the vertices with no neighbours at the end of the queue. Then colour the vertices in the queue with colours $1,\ldots,\omega-1$ in order. Colour each component of $A_2$ this way. Then $A_2$ and $B_{4,5}$ do not conflict. By~\ref{item:6}, $A_2$ and $B_{2,3}$ do not conflict.

$\bullet$ By~\ref{item:8}, each $C_{i,i+2}$ can be coloured with one colour. Colour $C_{1,3},C_{2,4},C_{3,5},C_{4,1},C_{5,2}$ with colours $2,\omega-1,\omega,\omega-2,3$ in order. By~\ref{item:12}, $B$ and $C$ do not conflict. By~\ref{item:11}, $C_{1,3}$ is anti-complete to each nontrivial component of $A_2$. (Note that each component of $A_2$ of size $j$ is coloured with colours $1,\ldots,j$.) By~\ref{item:25}, $C_{5.2}$ is anti-complete to $A_2$.

Suppose that $a\in A_2$ with colour $\omega-1$ is adjacent to $c\in C_{2,4}$. Then $a$ belongs to a component $K$ in $A_2$ of size $\omega-1$. By~\ref{item:9}, $K\cup\{v_2,c\}$ is a clique of size $\omega+1$.

Suppose that $a_1\in A_2$ with colour $\omega-2$ is adjacent to $c\in C_{4,1}$. Let $b$ be the vertex of colour $3$ in $B_{4,5}$. By the way we colour $A_2$, $a_1b\notin E$, and there is a vertex $a_2\in A_2$ adjacent to $a_1$, such that $a_2b\notin E$. (Note that every vertex in $A_2$ which is adjacent to $b$ is coloured with colour $1$.) By~\ref{item:10}, $a_2c\notin E$. Then $\{a_2,a_1,c,v_1,v_5,b\}$ induces a $P_6$. So $C$ and $A$ do not conflict.

$\bullet$ Colour $F_1,\ldots,F_5$ with colours $\omega-1,\omega-1,\omega,\omega-1,\omega-1$ in order. By~\ref{item:13}, this is a proper colouring of $F$. By~\ref{item:18}, $F$ and $C$ do not conflict. By~\ref{item:17}, $B_{2,3}$ and $F\setminus F_5$ do not conflict. Note that if $f\in F_5$ and $b\in B_{2,3}$ have the same colour $\omega-1$, then $B_{2,3}\cup\{f,v_2,v_3\}$ is a clique of size $\omega+1$ by~\ref{item:17}. So $F$ and $B_{2,3}$ do not conflict. Similarly, $F$ and $B_{4,5}$ do not conflict.

By~\ref{item:14}, $F_4$ and $F_5$ do not conflict with $A_2$. By~\ref{item:15}, $F_2$ and vertices in $A_2$ of colour $\omega-1$ do not conflict, or else there is a clique of size $\omega+1$ in $F_2\cup A_2\cup\{v_2\}$. Suppose that $a_1\in A_2$ with colour $\omega-1$ is adjacent to $f\in F_1$. Let $b$ be the vertex of colour 3 in $B_{4,5}$. By the way we colour $A_2$, we have $a_1b\notin E$, and there is a vertex $a_2\in A_2$ adjacent to $a_1$ such that  $a_2b\notin E$. By~\ref{item:16}, $a_2f\notin E$. Then $\{a_2,a_1,f,v_1,v_5,b\}$ induces a $P_6$. So $A$ and $F$ do not conflict. Now we have a proper colouring of $V\setminus Z$, so that the colour 1 is not used in $C\cup F$, and $F$ is coloured with $\omega-1$ and $\omega$.

Note that in the above proof, the fact that both $B_{2,3}$ and $B_{4,5}$ contain at least two vertices is only used to ensure that $A=A_2$. To prove that the colouring is indeed a proper colouring, we do not require $B_{2,3}$ or $B_{4,5}$ to contain at least two vertices. We state this so that some other cases can reduce to this case.

\case{1.2} One of $B_{2,3}$ and $B_{4,5}$ contains at least two vertices while the other contains exactly one vertex.

By symmetry we assume that $B_{2,3}$ contains at least two vertices. By~\ref{item:20}, $A_1$ and $A_4$ are empty. By~\ref{item:24}, one of $A_2\cup A_3$ and $A_5$ is empty.

\case{1.2.1} $A_5$ is empty.

By~\ref{item:19}, $A_3$ is stable. Based on the colouring of Case 1.1, colour $B_{4,5}$ with colour 3 and $A_3$ with colour $1$. By~\ref{item:2} and~\ref{item:6}, $A_3$ do not conflict with $A_2$ or $B_{2,3}$. This gives a proper colouring.

\case{1.2.2} $A_5$ is nonempty.

By symmetry it is equivalent to the case that $B_{4,5}$ contains at least two vertices while $B_{2,3}$ contains exactly one vertex, and $A_2$ is nonempty while $A_4\cup A_5$ is empty. Then it reduces to Case 1.1.

\case{1.3} Both $B_{2,3}$ and $B_{4,5}$ contain exactly one vertex.

\case{1.3.1} One of $A_2$ and $A_5$ is nonstable.

By symmetry we assume that $A_2$ is nonstable. By~\ref{item:19}, $A_1$ and $A_3$ are stable while $A_4$ and $A_5$ are empty. If $A_1$ is empty, then it reduces to Case 1.2.1. Now suppose that $A_1$ is nonempty.

$\bullet$ Colour $v_1,\ldots,v_5$ with colours $2,\omega,1,3,1$ in order.

$\bullet$ Note that $b\in B_{4,5}$ has at most one neighbour in each component of $A_2$. Colour each component of $A_2$ with colours in $\{1,\ldots,\omega-1\}$ so that the neighbour of $b$ is coloured with colour 1, if it exists. Colour $A_1$ with colour $1$ and $A_3$ with colour $2$.

$\bullet$ Colour $B_{2,3}$ with colour $2$ and $B_{4,5}$ with colour $\omega-2$.

$\bullet$ Colour $C_{1,3},C_{2,4},C_{3,5},C_{4,1},C_{5,2}$ with colours $3,\omega-1,\omega-2,\omega,\omega-1$ in order. By~\ref{item:23}, $C_{2,4}$ and $C_{5,2}$ do not conflict.

Suppose that $a_1\in A_2$ with colour $\omega-2$ is adjacent to $c\in C_{3,5}$. Let $b\in B_{4,5}$. Then $a_1b\notin E$, and there is a vertex $a_2\in A_2$ adjacent to $a_1$ such that $a_2c,a_2b\notin E$. Then $\{a_2,a_1,c,v_3,v_4,b\}$ induces a $P_6$. So $A$ and $C_{3,5}$ do not conflict. The reason why $C_{1,3}\cup C_{2,4}\cup C_{5,2}$ does not conflict with $A$ is similar to Case 1.1. So $C$ and $A$ do not conflict.

$\bullet$ Colour $F_1,\ldots,F_5$ with colours $\omega-1,\omega,\omega-1,\omega-1,\omega$ in order. The reason why $F$ does not conflict with $A$ is similar to Case 1.1.

\case{1.3.2} Both $A_2$ and $A_5$ are stable.

By~\ref{item:19}, $A_1$, $A_3$ and $A_4$ are stable.

\case{1.3.2.1} $A_1$ is nonempty.

$\bullet$ Colour $v_1,\ldots,v_5$ with colours $4,2,1,4,3$ in order.

$\bullet$ Colour $A_1,\ldots,A_5$ with colours $1,1,3,2,2$ in order.

$\bullet$ Colour $B_{2,3}$ with colour $3$ and $B_{4,5}$ with colour $2$.

$\bullet$ Colour $C_{1,3},C_{2,4},C_{3,5},C_{4,1},C_{5,2}$ with colours $5,3,4,6,4$ in order. By (22), $C_{2,4}$ and $A_3$ do not conflict. By (23), $C_{3,5}$ and $C_{5,2}$ do not conflict.

$\bullet$ Colour $F_1,\ldots,F_5$ with colours $6,6,6,6,5$ in order.

\case{1.3.2.2} $A_1$ is empty while $A_3\cup A_4$ is nonempty.

By symmetry we assume that $A_3$ is nonempty.

$\bullet$ Colour $v_1,\ldots,v_5$ with colours $3,4,1,3,2$ in order.

$\bullet$ Colour $A_2,A_3,A_4,A_5$ with colours $2,2,1,1$ in order.

$\bullet$ Colour $B_{2,3}$ with colour $2$ and $B_{4,5}$ with colour $1$.

$\bullet$ Colour $C_{1,3},C_{2,4},C_{3,5},C_{4,1},C_{5,2}$ with colours $4,5,6,5,3$ in order. By~\ref{item:23}, $C_{2,4}$ and $C_{4,1}$ do not conflict.

$\bullet$ Colour $F_1,\ldots,F_5$ with colours $6,6,6,5,6$ in order.

\case{1.3.2.3} $A_1\cup A_3\cup A_4$ is empty.

If one of $A_2$ and $A_5$ is empty, then it reduces to Case 1.1. Now suppose that both $A_2$ and $A_5$ are nonempty.

Now we prove that each vertex in $C_{2,4}\cup C_{3,5}$ is complete to one of $A_2$ and $A_5$, and anti-complete to the other. By symmetry, it suffices to show that $C_{2,4}$ is complete to one of $A_2$ and $A_5$ and anti-complete to the other. Let $a_1\in A_2$, $a_2\in A_5$, $b\in B_{2,3}$, $c\in C_{2,4}$. By~\ref{item:3}, $a_1a_2\in E$. By~\ref{item:6}, $ba_1\notin E$. By~\ref{item:12}, $bc\notin E$. By~\ref{item:24}, $ba_2\in E$. If $a_1c,a_2c\in E$, then \{$a_1,a_2,c,v_2$\} induces a diamond. If $a_1c,a_2c\notin E$, then \{$c,v_4,v_3,b,a_2,a_1$\} induces a $P_6$.

$\bullet$ Colour $v_1,\ldots,v_5$ with colours $1,3,5,4,3$ in order.

$\bullet$ Colour $A_2$ with colour $1$ and $A_5$ with colour $2$.

$\bullet$ Colour $B_{2,3}$ with colour $1$ and $B_{4,5}$ with colour $2$.

$\bullet$ Colour $C_{1,3},C_{4,1},C_{5,2}$ with colours $6,3,5$ in order. Colour the vertices in $C_{2,4}\cup C_{3,5}$ which is adjacent to $A_2$ with colour $2$, and those adjacent to $A_5$ with colour $1$. Suppose that $c_1\in C_{2,4}$ and $c_2\in C_{3,5}$ are adjacent and have the same colour, say $2$, then both of them are adjacent to some vertex $a\in A_2$, and so \{$a,c_1,c_2,v_2$\} induces a diamond. So $C_{2,4}$ and $C_{3,5}$ do not conflict.

$\bullet$ Colour $F_1,\ldots,F_5$ with colours $6,5,6,6,6$ in order.

\case{2} Exactly one of $B_{2,3}$ and $B_{4,5}$ is nonempty.

By symmetry we assume that $B_{2,3}$ is nonempty. By~\ref{item:19}, \{$A_1,\ldots,A_5$\} contains at most two nonstable sets. If there are two nonstable sets, they must be $A_2$ and $A_3$; if there is one nonstable set, it can be $A_2$, $A_3$ or $A_5$.

\begin{claim}\label{clm:A5C52}
Suppose that $B_{2,3}$ is nonempty. If  $a\in A_5$ and $c\in C_{3,5}\cup C_{5,2}$ are adjacent, then $G$ is $\omega$-colourable.
\end{claim}
\begin{proof}
By symmetry suppose that $c\in C_{5,2}$. Let $Q'=\{c,v_2,v_3,v_4,v_5\}$, then $Q'$ induces a $C_5$. By~\ref{item:12}, $c$ is anti-complete to $B_{2,3}$. Let $b\in B_{2,3}$, then $N_{Q'}(b)=\{v_2,v_3\}$ and $N_{Q'}(a)=\{c,v_5\}$. Then it reduces to Case 1.
\end{proof}

So in Case 2 we assume that $A_5$ and $C_{3,5}\cup C_{5,2}$ are anti-complete.

\case{2.1} Both $A_2$ and $A_3$ are nonstable.

By~\ref{item:19}, $A_1$, $A_4$ and $A_5$ are empty.

$\bullet$ Colour $v_1,\ldots,v_5$ with colours $1,\omega,1,2,3$ in order.

$\bullet$ Colour $A_3$ with colours in \{$2,\ldots,\omega$\}. For each component of $A_2$, colour it with colours in $\{1,\ldots,\omega-1\}$ using the largest colours available.

$\bullet$ Colour $C_{3,5},C_{4,1},C_{5,2}$ with colours $\omega,3,1$ in order. For each vertex in $C_{1,3}$, colour it with colour 2 if it has a neighbour in $C_{3,5}$, otherwise colour it with colour $\omega$. For each vertex in $C_{2,4}$, colour it with colour $\omega-1$ if it has a neighbour in $C_{5,2}$, otherwise colour it with colour 1.

Suppose that $a\in A_2$ and $c\in C_{2,4}$ are adjacent and have the same colour 1. Then $a$ is in a component $K$ in $A_2$ of size $\omega-1$. Then $K\cup\{v_2,c\}$ is a clique of size $\omega+1$ by~\ref{item:9}. Suppose that $a\in A_2$ and $c_1\in C_{2,4}$ are adjacent and have the same colour $\omega-1$. Then $c_1$ has a neighbour $c_2\in C_{5,2}$. By~\ref{item:25}, $ac_2\notin E$. Then $\{a,v_2,c_1,c_2\}$ induces a diamond. So $A_2$ and $C_{2,4}$ do not conflict. Similarly, $A_3$ and $C_{1,3}$ do not conflict.

Now we prove that $C_{4,1}$ is anti-complete to every nontrivial component of $A_2$ and $A_3$. By symmetry suppose that $a_1,a_2\in A_2$ such that $a_1a_2\in E$, and $a_1$ has a neighbour $c\in C_{4,1}$. By~\ref{item:10}, $a_2c\notin E$, and there is a vertex $a_3\in A_3$ such that $a_3c\notin E$ since $A_3$ is nonstable. Then $\{a_2,a_1,c,v_4,v_3,a_3\}$ induces a $P_6$. So $A$ and $C$ do not conflict.

$\bullet$ Now we prove that there is at most one vertex in $B_{2,3}$ which has a neighbour in $C_{4,1}$. Since every vertex in $C_{4,1}$ has at most one neighbour in $B_{2.3}$ by~\ref{item:12}, we may assume that $b_1,b_2\in B_{2,3}$, $c_1,c_2\in C_{4,1}$ such that $b_1c_1,b_2c_2\in E$. Let $a$ be a vertex in a nontrivial component of $A_3$. Then $ac_1,ac_2\notin E$, and so $\{c_2,v_1,c_1,b_1,v_3,a\}$ induces a $P_6$.

Colour $B_{2,3}$ with colours in $\{2,\ldots,\omega-1\}$ so that the vertex which has a neighbour in $C_{4,1}$ is coloured with colour 2, if it exists. So $B$ and $C$ do not conflict.

$\bullet$ Colour $F_1,\ldots,F_5$ with colours $\omega,\omega-1,\omega,\omega-1,\omega-1$ in order.

\case{2.2} Exactly one of $A_2$ and $A_3$ is nonstable.

By symmetry we assume that $A_2$ is nonstable. By~\ref{item:19}, $A_4$ and $A_5$ is empty.

$\bullet$ Colour $v_1,\ldots,v_5$ with colours $1,\omega,2,1,2$ in order.

$\bullet$ Colour $A_2$ with colours in $\{1,\ldots,\omega-1\}$, $A_1$ with colour 2 and $A_3$ with colour 1.

$\bullet$ Colour $B_{2,3}$ with colours in $\{1,3,4,\ldots,\omega-1\}$.

$\bullet$ Colour $C_{1,3},C_{2,4},C_{3,5}$ with colours $3,\omega-1,\omega$ in order. If $A_1$ is empty, then colour $C_{4,1}$ with colour 2 and $C_{5,2}$ with colour $\omega-2$, otherwise colour $C_{4,1}$ with colour $\omega-2$ and $C_{5,2}$ with colour $\omega-1$.

By~\ref{item:23}, $C_{2,4}$ and $C_{5,2}$ do not conflict. If $A_1$ is nonempty, then $B_{2,3}$ contains only one vertex with colour 1 by~\ref{item:20}. So $B_{2,3}$ and $C_{4,1}$ do not conflict. Suppose that $a_1\in A_2$ and $c\in C_{4,1}$ are adjacent and have the same colour 2 or $\omega-2$. Then there is a vertex $a_2\in A_2$ such that $a_1a_2\in E$ and $a_2c\notin E$. Let $b\in B_{2,3}$. Then $\{a_2,v_2,b,c,v_4,v_5\}$ or $\{a_2,a_1,c,v_4,v_3,b\}$ induces a $P_6$, depending on whether $b$ and $c$ are adjacent. So $A_2$ and $C_{4,1}$ do not conflict.

$\bullet$ Colour $F_1,\ldots,F_5$ with colours $\omega,\omega-1,\omega,\omega-1,\omega-1$ in order.

\case{2.3} $A_5$ is nonstable.

By~\ref{item:19}, $A_2$ and $A_3$ are empty.

$\bullet$ Colour $v_1,\ldots,v_5$ with colours $\omega-2,1,2,1,\omega$ in order.

$\bullet$ Colour $B_{2,3}$ with colours in $\{3,\ldots,\omega\}$. Exceptionally, when $B_{2,3}$ contains exactly $\omega-3$ vertices, we use colours $\{3,4,\ldots,\omega-2,\omega\}$ instead of colours $\{3,\ldots,\omega-1\}$.

$\bullet$ For each component of $A_5$, sort its vertices by the colours of their neighbours in $B_{2,3}$ in ascending order, putting the vertices with no neighbours in $B_{2,3}$ at the end of the queue. Then colour the vertices in the queue with colours $1,\ldots,\omega-1$ in order. Colour $A_1$ with colour 1 and $A_4$ with colour 2.

$\bullet$ Colour $C_{1,3},C_{2,4},C_{5,2}$ with colours $3,\omega,\omega-1$ in order. If $A_4$ is empty, then colour $C_{3,5}$ with colour $\omega-2$ and $C_{4,1}$ with colour 2, otherwise colour $C_{3,5}$ with colour 3 and $C_{4,1}$ with colour $\omega-2$.

By~\ref{item:23}, $C_{1,3}$ and $C_{3,5}$ do not conflict. If $A_4$ is nonempty, then $B_{2,3}$ contains only one vertex with colour 3 by~\ref{item:20}. So $B_{2,3}$ and $C_{4,1}$ do not conflict. Note that we assume that $A_5$ and $C_{3,5}\cup C_{5,2}$ are anti-complete. The reason why $A_5$ and $C_{1,3}$ do not conflict is similar to the reason why $A_2$ and $C_{4,1}$ do not conflict in Case 1.1. So $A_5$ and $C$ do not conflict.

$\bullet$ Colour $F_1,\ldots,F_5$ with colours $\omega,\omega-1,\omega-1,\omega,\omega-1$ in order.

\case{2.4} None of $A_1,\ldots,A_5$ is nonstable.

$\bullet$ Colour $v_1,\ldots,v_5$ with colours $1,3,2,1,3$ in order.

$\bullet$ Colour $A_1,\ldots,A_5$ with colours $2,1,1,4,2$ in order.

$\bullet$ Colour $B_{2,3}$ with colours in $\{1,4,5,\ldots,\omega\}$. Note that if $A_4$ is nonempty, then $B_{2,3}$ contains only one vertex with colour 1 by~\ref{item:20}. So $A_4$ and $B_{2,3}$ do not conflict.

$\bullet$ Colour $C_{2,4},C_{3,5},C_{4,1}$ with colours $5,6,3$ in order. If $A_1$ is empty, then colour $C_{5,2}$ with colour 2, otherwise colour $C_{5,2}$ with colour 5. If $A_4$ is empty, then colour $C_{1,3}$ with colour 4, otherwise colour $C_{1,3}$ with colour 6.

Note that we assume that $A_5$ and $C_{5,2}$ are anti-complete. By~\ref{item:23}, $C_{1,3}$ and $C_{3,5}$, $C_{2,4}$ and $C_{5,2}$ do not conflict.

$\bullet$ Colour $F_1,\ldots,F_5$ with colours $\omega,5,\omega,5,\omega$ in order.

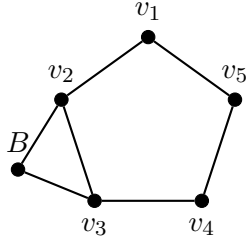
\begin{figure}[h!]
\centering
\begin{tikzpicture}[scale=0.6]
\tikzstyle{vertex}=[circle, draw, fill=black, inner sep=1pt, minimum size=5pt]
    \node[vertex, label=${v_1}$](1) at (0,2) {};
    \node[vertex, label=${v_2}$](2) at (-{2*cos(18)},{2*sin(18}) {};
    \node[vertex, label=below:${v_3}$](3) at (-{2*sin(36)},-{2*cos(36}) {};
    \node[vertex, label=below:${v_4}$](4) at ({2*sin(36)},-{2*cos(36}) {};
    \node[vertex, label=${v_5}$](5) at ({2*cos(18)},{2*sin(18}) {};
    \node[vertex, label=$B$](6) at (-{3*cos(18)},-{3*sin(18}) {};

    \Edge(1)(2)
    \Edge(2)(3)
    \Edge(3)(4)
    \Edge(4)(5)
    \Edge(1)(5)
    \Edge(2)(6)
    \Edge(3)(6)

\end{tikzpicture}
\caption{The graph series $S_n$. $B$ is a clique of size $n$ such that $B$ is complete to $\{v_2,v_3\}$ and anti-complete to $\{v_1,v_4,v_5\}$.}\label{fig:Sn}
\end{figure}

\case{3} $B$ is empty.

Let $S_1$ be the graph with $6$ vertices obtained from $C_5$ by adding a vertex adjacent to exactly two adjacent vertices of the $C_5$ (see \autoref{fig:Sn}). If $G$ contains an $S_1$, then it reduces to Case 1 or Case 2. Now assume that $G$ is $S_1$-free. We have the following two results:

\begin{claim}\label{clm:28}
If $G$ is $S_1$-free, then $C_{i,i+2}$ is anti-complete to $A_i\cup A_{i+2}$.
\end{claim}
\begin{proof}
By symmetry, it suffices to show that $C_{i,i+2}$ is anti-complete to $A_i$. If $a\in A_i$ and $c\in C_{i,i+2}$ are adjacent, then \{$v_i,c,v_{i+2},v_{i+3},v_{i+4},a$\} induces an $S_1$.
\end{proof}

\begin{claim}\label{clm:29}
If $G$ is $S_1$-free and $A_{i+1}$, $A_{i+3}$ and $A_{i+4}$ are all nonempty, then each vertex in $C_{i,i+2}$ is complete to one of $A_{i+1}$ and $A_{i+3}\cup A_{i+4}$, and anti-complete to the other.
\end{claim}
\begin{proof}
Let $a_1\in A_{i+1}$, $a_2\in A_{i+3}$, $a_3\in A_{i+4}$, $c\in C_{i+2}$. By~\ref{item:2}, $a_2a_3\notin E$. By~\ref{item:3}, $a_1a_2,a_1a_3\in E$. If $a_1c,a_2c,a_3c\in E$, then \{$a_1,a_2,a_3,c$\} induces a diamond. If $a_1c,a_2c,a_3c\notin E$, then \{$c,v_i,v_{i+4},a_3,a_1,a_2$\} induces a $P_6$. If $a_1c\notin E$, and exactly one of $a_2$ and $a_3$, say $a_3$, is adjacent to $c$, then \{$v_i,c,a_3,a_1,a_2,v_{i+3}$\} induces a $P_6$. If $a_1c\in E$, and exactly one of $a_2$ and $a_3$, say $a_2$, is adjacent to $c$ then \{$a_3,a_1,a_2,v_{i+3},v_{i+4},c$\} induces an $S_1$.
\end{proof}

By~\ref{item:19}, \{$A_1,\ldots,A_5$\} contains at most two nonstable sets, and $A_i$ and $A_j$ are nonstable only if $|i-j|=1$.

\case{3.1} \{$A_1,\ldots,A_5$\} contains two nonstable sets.

By symmetry, we assume that $A_2$ and $A_3$ are nonstable. By~\ref{item:19}, $A_1\cup A_4\cup A_5$ is empty.

$\bullet$ Colour $v_1,\ldots,v_5$ with colours $1,\omega,1,2,3,$ in order.

$\bullet$ Colour $A_2$ with colours in $\{1,\ldots,\omega-1\}$. By~\ref{item:16}, each component of $A_3$ contains at most two vertices which have a neighbour in $F_2\cup F_4$. Colour $A_3$ with colours in $\{2,\ldots,\omega\}$ so that the vertices which have a neighbour in $F_2\cup F_4$ are coloured with colours in $\{2,3\}$ if exist.

$\bullet$ Colour $C_{1,3},C_{2,4},C_{3,5},C_{4,1},C_{5,2}$ with colours $2,\omega-1,\omega,3,1$ in order. By \autoref{clm:28}, $C_{2,4}\cup C_{5,2}$ and $A_2$, $C_{1,3}\cup C_{3,5}$ and $A_3$ do not conflict. The reason why $C_{4,1}$ and $A$ do not conflict is similar to Case 2.1. So $C$ and $A$ do not conflict.

$\bullet$ Colour $F_1,\ldots,F_5$ with colours $\omega,\omega-1,\omega,\omega-1,\omega-1$ in order.

\case{3.2} \{$A_1,\ldots,A_5$\} contains exactly one nonstable set.

By symmetry, assume that $A_2$ is nonstable. By~\ref{item:19}, $A_4\cup A_5$ is empty.

$\bullet$ Colour $v_1,\ldots,v_5$ with colours $3,\omega,1,2,1$ in order.

$\bullet$ Colour $A_2$ with colours in $\{1,\ldots,\omega-1\}$, $A_1$ with colour $1$, $A_3$ with colour $2$.

$\bullet$ Colour $C_{1,3},C_{2,4},C_{3,5},C_{4,1},C_{5,2}$ with colours $2,\omega-1,\omega-2,\omega,3$ in order. By \autoref{clm:28}, $A_2$ and $C_{2,4}\cup C_{5,2}$, $A_3$ and $C_{1,3}$ do not conflict.

Suppose that $a_1\in A_2$ with colour $\omega-2$ is adjacent to $c\in C_{3,5}$. Then there is a vertex $a_2\in A_2$ such that $a_2a_1\in E$ and $a_2c\notin E$. Then $\{a_1,c,v_5,v_1,v_2,a_2\}$ induces an $S_1$. So $C$ and $A$ do not conflict.

$\bullet$ Colour $F_1,\ldots,F_5$ with colours $\omega,\omega-1,\omega,\omega-1,\omega-1$ in order.

\case{3.3} \{$A_1,\ldots,A_5$\} contains no nonstable sets.

\case{3.3.1} One of \{$A_1,\ldots,A_5$\} is empty.

By symmetry, assume that $A_1$ is empty.

$\bullet$ Colour $v_1,\ldots,v_5$ with colours $1,3,2,1,2$ in order.

$\bullet$ Colour $A_2,A_3,A_4,A_5$ with colours $1,1,2,3$ in order.

$\bullet$ Colour $C_{1,3},C_{2,4},C_{3,5},C_{4,1},C_{5,2}$ with colours $6,2,3,4,5$ in order. By \autoref{clm:28}, $C_{2,4}$ and $A_4$, $C_{3,5}$ and $A_5$ do not conflict.

$\bullet$ Colour $F_1,\ldots,F_5$ with colours $6,5,6,6,6$ in order.

\case{3.3.2} None of \{$A_1,\ldots,A_5$\} are empty.

$\bullet$ Colour $v_1,\ldots,v_5$ with colours $3,2,4,1,4$ in order.

$\bullet$ Colour $A_1,\ldots,A_5$ with colours $1,1,3,2,2$ in order.

$\bullet$ Colour $C_{2,4},C_{4,1},C_{5,2}$ with colours $4,5,6$ in order. By \autoref{clm:28}, $C_{1,3}$ and $A_1$, $C_{3,5}$ and $A_5$ do not conflict. By \autoref{clm:29}, each vertex in  $C_{1,3}$ is complete to one of $A_2$ and $A_4\cup A_5$ and anti-complete to the other, and each vertex in  $C_{3,5}$ is complete to one of $A_1\cup A_2$ and $A_4$ and anti-complete to the other. Colour $C_{1,3}$ and $C_{3,5}$ with colours in $\{1,2\}$ so that $C$ and $A$ do not conflict. Suppose that $c_1\in C_{1,3}$, $c_2\in C_{3,5}$, $c_1c_2\in E$. If $c_1$ and $c_2$ have the same colour, say 1, then both of them are adjacent to some vertex $a\in A_4$, and so \{$a,c_1,c_2,v_3$\} induces a diamond. So $C_{1,3}$ and $C_{3,5}$ do not conflict.

$\bullet$ Colour $F_1,\ldots,F_5$ with colours $5,6,6,6,6$ in order.

This completes the proof of \autoref{thm:imperfect}.
\end{proof}

\section{Computing the Chromatic Number of ($P_6$, diamond)-Free Graphs in Polynomial Time}
\label{sec:poly}

In this section we prove the following \autoref{thm:poly}:

\begin{theorem}\label{thm:poly}
The chromatic number of ($P_6$,diamond)-free graphs can be computed in polynomial time.
\end{theorem}

This answers an open question from~\cite{CHM18}. To prove \autoref{thm:poly}, we use \autoref{thm:K6free} below.

\begin{theorem}\label{thm:K6free}
The chromatic number of ($P_6$,diamond,$K_6$)-free graphs can be computed in polynomial time.
\end{theorem}

To prove \autoref{thm:K6free}, we use the following three theorems.

\begin{theorem}\label{thm:omega3}
There is one 6-vertex-critical ($P_6$,diamond)-free graph with clique number 3.
\end{theorem}

\begin{theorem}\label{thm:omega4}
There are no 6-vertex-critical ($P_6$,diamond)-free graphs with clique number 4.
\end{theorem}

\begin{theorem}\label{thm:omega5}
There are no 6-vertex-critical ($P_6$,diamond)-free graphs with clique number 5.
\end{theorem}

\begin{figure}[h!]
\centering
\begin{tikzpicture}[scale=4.5]
\tikzstyle{vertex}=[circle, draw, fill=black, inner sep=1pt, minimum size=5pt]
        \node[vertex](0) at (0.1736481777, -0.984807753) {};
        \node[vertex](1) at (1.0, 0.0)  {};
        \node[vertex](2) at (0.5971585917, 0.8021231928) {};
        \node[vertex](3) at (-0.9396926208, -0.3420201433) {};
        \node[vertex](4) at (0.8936326403, -0.4487991802) {};
        \node[vertex](5) at (0.8936326403, 0.4487991802) {};
        \node[vertex](6) at (0.396079766, 0.9182161069) {};
        \node[vertex](7) at (-0.6862416379, -0.7273736416) {};
	    \node[vertex](8) at (0.9730448706, 0.2306158707) {};
	    \node[vertex](9) at (-0.9396926208, 0.3420201433) {};
	    \node[vertex](10) at (0.7660444431, 0.6427876097) {};
	    \node[vertex](11) at (0.1736481777, 0.984807753) {};
	    \node[vertex](12) at (-0.8354878114, 0.5495089781) {};
	    \node[vertex](13) at (-0.5, 0.8660254038) {};
	    \node[vertex](14) at (0.5971585917, -0.8021231928) {};
	    \node[vertex](15) at (-0.0581448289, 0.9983081583) {};
	    \node[vertex](16) at (-0.5, -0.8660254038) {};
	    \node[vertex](17) at (-0.2868032327, 0.9579895123) {};
	    \node[vertex](18) at (-0.6862416379, 0.7273736416) {};
	    \node[vertex](19) at (0.9730448706, -0.2306158707) {};
	    \node[vertex, label=below:$y$](20) at (-0.2868032327, -0.9579895123) {};
	    \node[vertex](21) at (0.396079766, -0.9182161069) {};
	    \node[vertex](22) at (0.7660444431, -0.6427876097) {};
	    \node[vertex](23) at (-0.0581448289, -0.9983081583) {};
	    \node[vertex](24) at (-0.9932383577, 0.1160929141) {};
	    \node[vertex, label=left:$x$](25) at (-0.9932383577, -0.1160929141) {};
	    \node[vertex](26) at (-0.8354878114, -0.5495089781) {};
	   
    \Edge(0)(1)
 	\Edge(0)(2)
 	\Edge(0)(3)
 	\Edge(0)(4)
	\Edge(0)(5)
	\Edge(0)(6)
 	\Edge(0)(7)
 	\Edge(0)(8)
 	\Edge(0)(9)
	\Edge(0)(10)
	
	\Edge(1)(2)
 	\Edge(1)(11)
 	\Edge(1)(12)
 	\Edge(1)(13)
	\Edge(1)(14)
	\Edge(1)(15)
 	\Edge(1)(16)
 	\Edge(1)(17)
 	\Edge(1)(18)
 	
 	\Edge(2)(19)
 	\Edge(2)(20)
 	\Edge(2)(21)
 	\Edge(2)(22)
	\Edge(2)(23)
	\Edge(2)(24)
 	\Edge(2)(25)
 	\Edge(2)(26)
 	
	\Edge(3)(4)
 	\Edge(3)(11)
 	\Edge(3)(12)
 	\Edge(3)(13)
	\Edge(3)(14)
	\Edge(3)(19)
 	\Edge(3)(20)
 	\Edge(3)(21)
 	\Edge(3)(22)
	
	\Edge(4)(15)
 	\Edge(4)(16)
 	\Edge(4)(17)
 	\Edge(4)(18)
	\Edge(4)(23)
	\Edge(4)(24)
 	\Edge(4)(25)
 	\Edge(4)(26)
 	
 	\Edge(5)(6)
 	\Edge(5)(11)
 	\Edge(5)(12)
 	\Edge(5)(15)
	\Edge(5)(16)
	\Edge(5)(19)
 	\Edge(5)(20)
 	\Edge(5)(23)
 	\Edge(5)(24)
 	
 	\Edge(6)(13)
 	\Edge(6)(14)
 	\Edge(6)(17)
 	\Edge(6)(18)
	\Edge(6)(21)
	\Edge(6)(22)
 	\Edge(6)(25)
 	\Edge(6)(26)
 	
 	\Edge(7)(8)
 	\Edge(7)(11)
 	\Edge(7)(13)
 	\Edge(7)(15)
	\Edge(7)(17)
	\Edge(7)(19)
 	\Edge(7)(21)
 	\Edge(7)(23)
 	\Edge(7)(25)
 	
 	\Edge(8)(12)
 	\Edge(8)(14)
 	\Edge(8)(16)
 	\Edge(8)(18)
	\Edge(8)(20)
	\Edge(8)(22)
 	\Edge(8)(24)
 	\Edge(8)(26)
 	
 	\Edge(9)(10)
 	\Edge(9)(11)
 	\Edge(9)(14)
 	\Edge(9)(16)
	\Edge(9)(17)
	\Edge(9)(20)
 	\Edge(9)(21)
 	\Edge(9)(23)
 	\Edge(9)(26)
 	
 	\Edge(10)(12)
 	\Edge(10)(13)
 	\Edge(10)(15)
 	\Edge(10)(18)
	\Edge(10)(19)
	\Edge(10)(22)
 	\Edge(10)(24)
 	\Edge(10)(25)
 
    \Edge(11)(18)
 	\Edge(11)(22)
 	\Edge(11)(24)
 	\Edge(11)(25)
	\Edge(11)(26)
	
	\Edge(12)(17)
 	\Edge(12)(21)
 	\Edge(12)(23)
 	\Edge(12)(25)
	\Edge(12)(26)
	
	\Edge(13)(16)
 	\Edge(13)(20)
 	\Edge(13)(23)
 	\Edge(13)(24)
	\Edge(13)(26)
	
	\Edge(14)(15)
 	\Edge(14)(19)
 	\Edge(14)(23)
 	\Edge(14)(24)
	\Edge(14)(25)
	
	\Edge(15)(20)
 	\Edge(15)(21)
 	\Edge(15)(22)
 	\Edge(15)(26)
	
	\Edge(16)(19)
 	\Edge(16)(21)
 	\Edge(16)(22)
 	\Edge(16)(25)
 	
 	\Edge(17)(19)
 	\Edge(17)(20)
 	\Edge(17)(22)
 	\Edge(17)(24)
 	
 	\Edge(18)(19)
 	\Edge(18)(20)
 	\Edge(18)(21)
 	\Edge(18)(23)
 	
 	\Edge(19)(26)
 	\Edge(20)(25)
 	\Edge(21)(24)
 	\Edge(22)(23)
	
\end{tikzpicture}
\caption{The complement of the 27-vertex Schl\"{a}fli graph. The unique 6-vertex-critical ($P_6$,diamond)-free graph $\mathcal{G}$ with clique number 3 is obtained by removing the vertices with labels $x$ and~$y$. }
\label{fig:sch}
\end{figure}
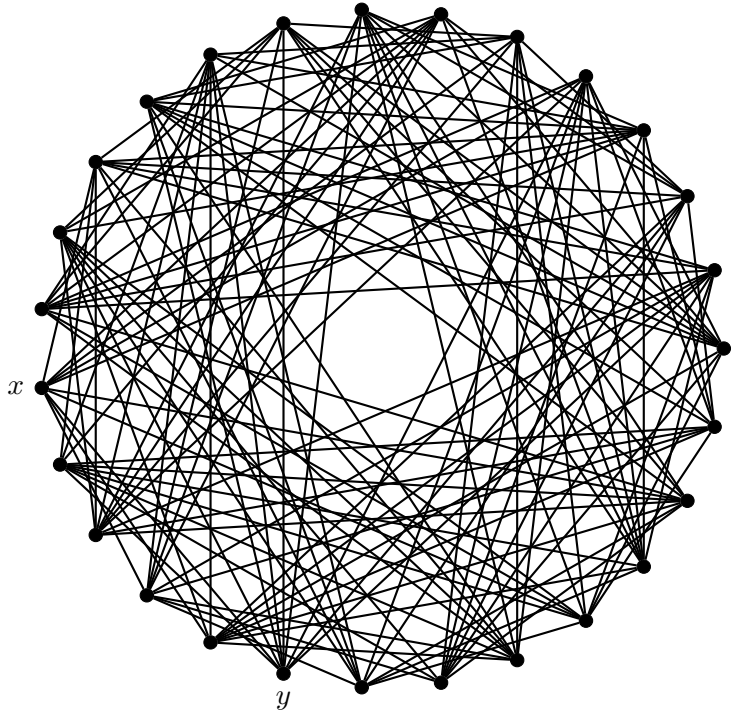

\begin{proof}[Proof of \autoref{thm:poly}]
Let $G$ be a ($P_6$,diamond)-free graph. We can calculate $\vartheta(\overline{G})$ in polynomial time~\cite{GLS84}. When $\vartheta(\overline{G})>5$, since $\omega(G)\leq\vartheta(\overline{G})\leq\chi(G)$~\cite{GLS84} and $\chi(G)\leq max\{\omega(G),6\}$ (i.e.~\autoref{thm:bound}), we have $\chi(G)=\lceil \vartheta(\overline{G}) \rceil$. Then we just have to deal with the case that $\vartheta(\overline{G})\leq5$, which implies that $\omega(G)\leq5$. Then the theorem follows from \autoref{thm:K6free}.
\end{proof}

\begin{proof}[Proof of \autoref{thm:K6free}]
By \autoref{thm:bound} and \autoref{thm:4CP6}, we only need to consider 5-colouring. By \autoref{thm:exhaustive}, we need to prove that there are a finite number of 6-vertex-critical ($P_6$,diamond,$K_6$)-free graphs. 
By \autoref{thm:P6C3}, every 6-vertex-critical ($P_6$,diamond)-free graph has clique number at least 3. Then the theorem follows from  \autoref{thm:omega3}, \autoref{thm:omega4}, and \autoref{thm:omega5}.
\end{proof}

The unique 6-vertex-critical ($P_6$,diamond)-free graph $\mathcal{G}$ with clique number 3 from \autoref{thm:omega3} has 25 vertices and can be obtained from the complement of the Schl\"{a}fli graph by deleting the vertices labelled $x$ and $y$ in Figure~\ref{fig:sch}. This graph can also be inspected at \textit{the House of Graphs}~\cite{BCGM13} at: \url{https://hog.grinvin.org/ViewGraphInfo.action?id=45613}.

The proof of \autoref{thm:omega3} uses computational methods. In the Appendix we give a computer-free proof of a weaker version of \autoref{thm:omega3} (i.e.\ \autoref{thm:omega3_manual}). This weaker theorem still suffices to give a complete computer-free proof of \autoref{thm:poly} and \autoref{thm:K6free}.

\begin{proof}[Proof of \autoref{thm:omega3}]
We used the generation algorithm for $k$-critical $H$-free graphs from~\cite{CGSZ20,GS18} and extended it to generate 6-vertex-critical ($P_6$,diamond,$K_4$)-free graphs. The algorithm terminated in less than 5 minutes and yielded the graph $\mathcal{G}$ as the only 6-vertex-critical ($P_6$,diamond,$K_4$)-free graph.
The source code of the program can be downloaded from~\cite{criticalpfree-site}. We refer to~\cite{CGSZ20,GS18} for more details on the algorithm and the proof of its correctness.
\end{proof}

In the first subsection we prove some lemmas which we will use in the next two subsections to prove \autoref{thm:omega4} and \autoref{thm:omega5}, and in the Appendix to prove \autoref{thm:omega3_manual}.

\subsection{Lemmas}\label{ssc:lemma}
\begin{lemma}\label{lem:Hfree-45}
Every 6-vertex-critical ($P_6$,diamond)-free graph with clique number $\omega$ ($\omega=4,5$) is $S_{\omega-2}$-free. (See \autoref{fig:Sn} for $S_n$.)
\end{lemma}

\begin{proof}
Let $G=(V,E)$ be a 6-vertex-critical ($P_6$,diamond)-free graph with clique number $\omega$ containing an $S_{\omega-2}$. In the proof of this lemma, we follow the partitioning of $V$ from \autoref{sec:imperfect}. Let the $C_5$ in the $S_{\omega-2}$ be $Q$, and the other two or three vertices in the $S_{\omega-2}$, denoted by $b_1,b_2$ (and $b_3$ if it exists), be in $B_{2,3}$. By $\omega(G)=\omega$, $B_{2,3}$ contains no vertices other than these two or three vertices, and $F_5$ is empty. By~\ref{item:20}, $A_1\cup A_4$ is empty. By~\ref{item:24}, one of $A_2\cup A_3$ and $A_5$ is empty. Assume that $B=B_{2,3}\cup B_{4,5}$. Then the edges between $B_{2,3}$ and $B_{4,5}$ form a matching.

\begin{claim}\label{clm:B45empty}
$B=B_{2,3}$.
\end{claim}
\begin{proof}
Suppose that $B_{4,5}$ is nonempty.

Now we prove that if $u_1\in B_{2,3}$ and $u_2\in B_{4,5}$ are nonadjacent, then both $u_1$ and $u_2$ have no neighbours in $C$. Assume that $c\in C_{4,1}$ is adjacent to $u_1$. Then $\{v_3,u_1,c,v_1,v_5,u_2\}$ induces a $P_6$. So $u_1$ and $C_{4,1}$ are anti-complete. Similarly, $u_2$ and $C_{1,3}$ are anti-complete. By~\ref{item:12}, both $u_1$ and $u_2$ have no neighbours in $C$.

We discuss two cases based on the value of $\omega$, and give a contradiction for each case.

\case{1} $\omega$=4.

Let $b_4\in B_{4,5}$. Then one of $b_1$ and $b_2$, say $b_1$, is nonadjacent to $b_4$. Then both $b_1$ and $b_4$ have no neighbours in $C$.

If $A_5$ is empty, we consider the degree of $b_1$. Since $b_1$ has two neighbours $v_2,v_3$ in $Q$, one neighbour $b_2$ in $B_{2,3}$, at most one neighbour in $B_{4,5}$, and no neighbours in $A\cup C\cup F$, we have $d(b_1)\leq4$.

If $A_2\cup A_3$ is empty, we consider the degree of $b_4$. Since $b_4$ has two neighbours $v_4,v_5$ in $Q$, at most one neighbour in $B_{2,3}$, at most one neighbour in $B_{4,5}\cup F_2$, and no neighbours in $A\cup C$, we have $d(b_4)\leq4$.

\case{2} $\omega$=5.

Now we prove that either $B_{4,5}$ contains at most one vertex, or $B_{2,3}$ and $B_{4,5}$ are anti-complete. Suppose that $b_4,b_5\in B_{4,5}$ such that $b_4b_1\in E$. We may assume that $b_5b_2\notin E$. Recall that if $u_1\in B_{2,3}$ and $u_2\in B_{4,5}$ are nonadjacent, then both $u_1$ and $u_2$ have no neighbours in $C$. Since both $B_{2,3}$ and $B_{4,5}$ contain at least two vertices, $B_{2,3}$ and $C_{4,1}$, $B_{4,5}$ and $C_{1,3}$ are anti-complete. If $c\in C_{4,1}$, then $\{b_2,b_1,b_4,v_5,v_1,c\}$ induces a $P_6$. So $C_{4,1}$ is empty. If $f\in F_3$, then $\{v_1,f,v_3,b_1,b_4,b_5\}$ induces a $P_6$. So $F_3$ is empty. Similarly, $C_{1,3}\cup F_4$ is empty. Then $d(v_1)\leq3$. So either $B_{4,5}$ contains at most one vertex, or $B_{2,3}$ and $B_{4,5}$ are anti-complete.

Let $b_4\in B_{4,5}$. Suppose that $b_4$ has no neighbours in $A_2\cup A_3$. If $B_{4,5}$ contains only one vertex, since $b_4$ has two neighbours $v_4,v_5$ in $Q$, at most one neighbour in $B_{2,3}$, at most one neighbour in $F_2$, and no neighbours in $A\cup C$, we have $d(b_4)\leq4$. If $B_{2,3}$ and $B_{4,5}$ are anti-complete, since $b_4$ has two neighbours $v_4,v_5$ in $Q$, at most two neighbours in $B_{4,5}\cup F_2$, and no neighbours in $B_{2,3}\cup A\cup C$,  we have $d(b_4)\leq4$. So $b_4$ has a neighbour in $A_2\cup A_3$. Then $A_5$ is empty. By symmetry we may assume that $b_1b_4\notin E$, then $b_1$ has no neighbour in $C_{4,1}$. Since $b_1$ has two neighbours $v_2,v_3$ in $Q$, two neighbours $b_2,b_3$ in $B_{2,3}$, no neighbours in $A\cup B_{4,5}\cup C$, we have $d(b_1)=4$.
\end{proof}

Then by $\delta(G)\geq5$, each vertex of $B_{2,3}$ has at least $6-\omega$ neighbours in $C_{4,1}\cup A_5$. Moreover, no pairs of vertices in $B_{2,3}$ can have a common neighbour in $C_{4,1}\cup A_5$ since $G$ is diamond-free.

\begin{claim}\label{clm:Sfree-1}
$A_2\cup A_3$ is empty.
\end{claim}
\begin{proof}
By symmetry suppose that $a\in A_2$. Then $A_5$ is empty and $b_1$ has a neighbour $c\in C_{4,1}$. Then $\{b_2,v_2,a,c,v_4,v_5\}$ or $\{a,v_2,b_1,c,v_4,v_5\}$ induces a $P_6$, depending on whether $a$ and $c$ are adjacent.
\end{proof}

\begin{claim}\label{clm:Sfree-2}
$C_{5,2}\cup C_{3,5}$ and $A_5$ are anti-complete.
\end{claim}
\begin{proof}
By symmetry, it suffices to show that $C_{5,2}$ and $A_5$ are anti-complete. Suppose that $c\in C_{5,2}$ and $a\in A_5$ are adjacent. Let $Q'=\{c,v_2,v_3,v_4,v_5\}$. Then $N_{Q'}(b_1)=N_{Q'}(b_2)=\{v_2,v_3\}$, $N_{Q'}(a)=\{c,v_5\}$, moreover $N_{Q'}(b_3)=\{v_2,v_3\}$ if $b_3$ exists. This contradicts with \autoref{clm:B45empty}.
\end{proof}

\begin{claim}\label{clm:Sfree-3}
If $a\in A_5$ has a neighbour in $B_{2,3}$, then $a$ has no neighbours in $A_5$.
\end{claim}
\begin{proof}
Suppose that $a_1,a_2\in A_5$ such that $a_1b_1,a_1a_2\in E$. Then $a_2b_1,a_1b_2\notin E$, moreover $a_1b_3\notin E$ if $b_3$ exists. Let $Q'=\{a_1,b_1,v_2,v_1,v_5\}$. Then $N_{Q'}(b_2)=N_{Q'}(v_3)=\{b_1,v_2\}$, $N_{Q'}(a_2)=\{a_1,v_5\}$, moreover $N_{Q'}(b_3)=\{b_1,v_2\}$ if $b_3$ exists. This contradicts with \autoref{clm:B45empty}.
\end{proof}

\begin{claim}\label{clm:Sfree-4}
If $a\in A_5$ has a neighbour $b\in B_{2,3}$, then $a$ is complete to $C_{4,1}\setminus N_{C_{4,1}}(b)$.
\end{claim}
\begin{proof}
Suppose that $c\in C_{4,1}$ such that $ac,bc\notin E$, then $\{c,v_4,v_5,a,b,v_2\}$ induces a $P_6$.
\end{proof}

\begin{claim}\label{clm:Sfree-5}
If $c\in C_{4,1}$ has a neighbour $b\in B_{2,3}$, then $c$ is complete to $A_5\setminus N_{A_5}(b)$, and $A_5$ is stable.
\end{claim}
\begin{proof}
Suppose that $a\in A_5$ such that $ac,ab\notin E$, then $\{v_3,b,c,v_1,v_5,a\}$ induces a $P_6$. So $c$ is complete to $A_5\setminus N_{A_5}(b)$. Suppose that $a_1,a_2\in A_5$ such that $a_1a_2\in E$. Then $a_1b,a_2b\notin E$ by \autoref{clm:Sfree-3}, and so $a_1c,a_2c\in E$. Now $\{a_1,a_2,c,v_5\}$ induces a diamond.
\end{proof}

\begin{claim}\label{clm:Sfree-6}
$C_{2,4}\cup C_{1,3}$ and $A_5$ are complete.
\end{claim}
\begin{proof}
By symmetry, it suffices to show that $C_{2,4}$ and $A_5$ are complete. Suppose that $c\in C_{2,4}$ and $a\in A_5$ are nonadjacent. Since $a$ has at most one neighbour in $B_{2,3}$, we may assume that $b_1a\notin E$. Then $\{b_1,v_2,c,v_4,v_5,a\}$ induces a $P_6$.
\end{proof}

\begin{claim}\label{clm:Sfree-7}
If there are at least two vertices in $B_{2,3}$ which have a neighbour in $C_{4,1}$, then $C_{4,1}$ and $C_{5,2}\cup C_{3,5}$ are complete.
\end{claim}
\begin{proof}
By symmetry, it suffices to show that $C_{4,1}$ and $C_{5,2}$ are complete. Let $c_1,c_2\in C_{4,1}$ such that $c_1b_1,c_2b_2\in E$, $c_3\in C_{5,2}$. Suppose that $c_3c_1,c_3c_2\notin E$, then $\{c_1,v_4,c_2,b_2,v_2,c_3\}$ induces a $P_6$. So $c_3$ is adjacent to at least one of $c_1$ and $c_2$. By symmetry assume that $c_1c_3\in E$. Let $c_4$ be an arbitrary vertex in $C_{4,1}\setminus\{c_1\}$ with $c_4c_3\notin E$. Then $\{c_4,b_2,b_1,c_1,c_3,v_5\}$ or $\{c_4,v_4,c_1,c_3,v_2,b_2\}$ induces a $P_6$, depending on whether $c_4$ and $b_2$ are adjacent.
\end{proof}

\begin{claim}\label{clm:Sfree-8}
If there are at least two vertices in $B_{2,3}$ which have a neighbour in $C_{4,1}$, then $C_{4,1}$ and $C_{2,4}\cup C_{1,3}$ are anti-complete.
\end{claim}
\begin{proof}
Let both $b_1$ and $b_2$ have a neighbour in $C_{4,1}$. By symmetry, it suffices to show that $C_{4,1}$ and $C_{2,4}$ are anti-complete. Let $c_1\in C_{2,4}$. Since $G$ is diamond-free, $c_1$ has at most one neighbour in $C_{4,1}$. Suppose that $c_2\in C_{4,1}$ is adjacent to $c_1$. Since $c_2$ is adjacent to at most one vertex in $\{b_1,b_2\}$, assume that $b_1c_2\notin E$. Let $c_3\in C_{4,1}$ be a neighbour of $b_1$. Then $\{c_1,c_2,v_1,c_3,b_1,v_3\}$ induces a $P_6$.
\end{proof}

Now we give a 5-colouring of $G$ as follows.

\case{1} There are at least two vertices in $B_{2,3}$ which have a neighbour in $C_{4,1}$.

By \autoref{clm:Sfree-5}, $A_5$ is stable.

$\bullet$ Colour $v_1,\ldots,v_5$ with colours $4,3,4,2,1$ in order.

$\bullet$ Colour $b_1,b_2,b_3$ with colours 1,2,5 in order.

$\bullet$ Colour $A_5$ with colour 4.

$\bullet$ Colour $C_{1,3},C_{2,4},C_{3,5},C_{4,1},C_{5,2}$ with colours $3,1,2,3,2$ in order. By \autoref{clm:Sfree-8}, $C_{1,3}$ and $C_{4,1}$ do not conflict. Note that $C_{4,1}$ contains a $2K_1$. So $C_{5,2}$ and $C_{3,5}$ are anti-complete by \autoref{clm:Sfree-7} since $G$ is diamond-free.

$\bullet$ Colour $F$ with colour 5. Note that $F_5$ is empty, so $F$ and $B_{2,3}$ do not conflict by~\ref{item:17}.

$\bullet$ Let $K$ be an arbitrary component of $Z$. If $K$ contains only one vertex, colour it with colour 4. If $K$ contains at least two vertices, then either $K$ is 3-colourable if $K$ has no neighbours in $C$, or $K\cup N_C(K)$ is a clique of size at most 4 if $K$ has a neighbour in $C$ (cf.\ the proof of \autoref{clm:Z colouring}). So $K$ can be coloured with colours in $\{1,2,3,4\}$. Colour each component of $Z$ this way.

\case{2} There is at most one vertex in $B_{2,3}$ which has a neighbour in $C_{4,1}$.

Let $b_1$ have no neighbours in $C_{4,1}$. Moreover if $\omega=5$, let $b_2$ have no neighbours in $C_{4,1}$.

$\bullet$ Colour $v_1,\ldots,v_5$ with colours $4,3,4,3,1$ in order.

$\bullet$ Colour $b_1,b_2,b_3$ with colours $1,2,5$ in order.

$\bullet$ Colour each nontrivial component of $A_5$ with colours in $\{2,3,4,5\}$. Colour each trivial component of $A_5$ with colour 3. By \autoref{clm:Sfree-3}, $A_5$ and $B_{2,3}$ do not conflict.

$\bullet$ Colour $C_{1,3},C_{2,4},C_{3,5},C_{4,1},C_{5,2}$ with colours $1,1,3,1,4$ in order. By \autoref{clm:Sfree-2}, $C_{5,2}\cup C_{3,5}$ and $A_5$ do not conflict.

If $\omega=4$, then $b_1$ has at least two neighbours in $A_5$. If $\omega=5$, then both $b_1$ and $b_2$ have at least one neighbour in $A_5$. So there are two nonadjacent vertices $a_1,a_2\in A_5$ which are complete to $C_{4,1}$ by \autoref{clm:Sfree-3} and \autoref{clm:Sfree-4}. Note that $\{a_1,a_2\}$ and $C_{2,4}\cup C_{1,3}$ are complete by \autoref{clm:Sfree-6}. Then $C_{4,1}$, $C_{2,4}$ and $C_{1,3}$ are pairwise anti-complete.

$\bullet$ Colour $F$ with colour 5. By~\ref{item:14}, $F_2\cup F_3$ is anti-complete to $A_5$. If $a_1\in A_5$ with colour 5 and $f\in F_1\cup F_4$ are adjacent, by symmetry assume that $f\in F_1$. Let $a_2\in A_5$ be adjacent to $a_1$, then $a_2f\notin E$ by~\ref{item:16}, and so $\{a_2,a_1,f,v_1,v_2,b_1\}$ induces a $P_6$. So $A_5$ and $F$ do not conflict.

$\bullet$ The colouring of $Z$ is similar to Case 1.

So $G$ is 5-colourable, a contradiction.
\end{proof}

\begin{figure}[h!]
\centering
\begin{tikzpicture}[scale=0.6]
\tikzstyle{vertex}=[circle, draw, fill=black, inner sep=1pt, minimum size=5pt]
    \node[vertex, label=${v_1}$](1) at (0,2) {};
    \node[vertex, label=${v_2}$](2) at (-{2*cos(18)},{2*sin(18}) {};
    \node[vertex, label=below:${v_3}$](3) at (-{2*sin(36)},-{2*cos(36}) {};
    \node[vertex, label=below:${v_4}$](4) at ({2*sin(36)},-{2*cos(36}) {};
    \node[vertex, label=${v_5}$](5) at ({2*cos(18)},{2*sin(18}) {};
    \node[vertex, label=${b_1}$](6) at (-{3*cos(18)},-{3*sin(18}) {};
    \node[vertex, label=${b_2}$](7) at ({3*cos(18)},-{3*sin(18}) {};

    \Edge(1)(2)
    \Edge(2)(3)
    \Edge(3)(4)
    \Edge(4)(5)
    \Edge(1)(5)
    \Edge(2)(6)
    \Edge(3)(6)
    \Edge(4)(7)
    \Edge(5)(7)

\end{tikzpicture}
\caption{The graph $D_1$.}\label{fig:D1}
\end{figure}
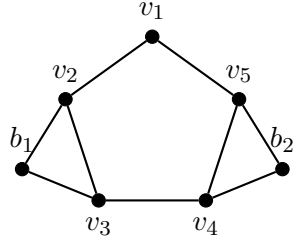

\begin{lemma}\label{lem:H4free}
Every 6-vertex-critical ($P_6$,diamond)-free graph with clique number 3 is $D_1$-free. (See \autoref{fig:D1} for $D_1$.)
\end{lemma}

\begin{proof}
Let $G=(V,E)$ be a 6-vertex-critical ($P_6$,diamond)-free graph with clique number 3 containing a $D_1$. In the proof of this lemma, we follow the partitioning of $V$ from \autoref{sec:imperfect}. Let the $C_5$ in the $D_1$ be $Q$, and the other two vertices $b_1\in B_{2,3}$, $b_2\in B_{4,5}$ such that $b_1b_2\notin E$. By $\omega(G)=3$, $B=\{b_1,b_2\}$.

\begin{claim}\label{clm:H4free-C13C41}
$C_{1,3}\cup C_{4,1}\cup A_3\cup A_4\cup F_2\cup F_5$ is empty.
\end{claim}
\begin{proof}
By symmetry, it suffices to show that $C_{4,1}\cup A_4\cup F_2$ is empty. Suppose that $c\in C_{4,1}$, then $\{v_3,b_1,c,v_1,v_5,b_2\}$ or $\{b_1,v_2,v_1,c,v_4,b_2\}$ induces a $P_6$, depending on whether $b_1$ and $c$ are adjacent. Suppose that $a\in A_4$, then $\{v_1,v_2,b_1,a,v_4,b_2\}$ induces a $P_6$. Suppose that $f\in F_2$, then $\{b_2,v_4,v_5,f\}$ induces a $K_4$ or a diamond.
\end{proof}

\begin{claim}\label{clm:H4free-A1}
$A_1$ contains at most one vertex.
\end{claim}
\begin{proof}
Suppose that $a_1,a_2\in A_1$. By the properties from \autoref{sec:imperfect}, $A_1$ is stable, complete to $B_{2,3}\cup B_{4,5}\cup C_{5,2}\cup C_{2,4}\cup C_{3,5}$ and anti-complete to $A_2\cup A_5\cup F_3\cup F_4\cup Z$.

Suppose that $f\in F_1$ such that $a_1f\in E$, then $a_2f\notin E$ as the graph is diamond-free, then $\{a_2,b_1,a_1,f,v_4,v_5\}$ induces a $P_6$. So $A_1$ is anti-complete to $F_1$. Then $a_1$ and $a_2$ are comparable.
\end{proof}

Then by $\delta(G)\geq5$, $b_1$ has at least two neighbours in $A_5$, and $b_2$ has at least two neighbours in $A_2$. Let $a_1,a_2\in N_{A_2}(b_2)$. By the properties from \autoref{sec:imperfect}, $A_2$ and $A_5$ are stable, and $A_2$ is complete to $A_5\cup B_{4,5}$ and anti-complete to $A_1\cup B_{2,3}\cup C_{5,2}\cup F_4\cup Z$. Now we prove that $a_1$ and $a_2$ are comparable by showing that their neighbourhoods in $C_{2,4}\cup C_{3,5}\cup F_1\cup F_3$ are equal.


$\bullet$ $\{a_1,a_2\}$ is anti-complete to $C_{2,4}$: Suppose that $c\in C_{2,4}$ such that $a_1c\in E$, then $c$ has no neighbours in $A_2\cup A_5$ other than $a_1$ because of the diamond-freeness. Let $a_3\in A_5$, then $\{a_2,a_3,a_1,c,v_4,v_3\}$ induces a $P_6$.

$\bullet$ For each vertex $c\in C_{3,5}$, either $a_1c,a_2c\in E$ or $a_1c,a_2c\notin E$: Suppose that $a_1c\in E$, $a_2c\notin E$. Let $a_3\in A_5$, then $a_3c\notin E$, then $\{a_2,a_3,a_1,c,v_3,v_4\}$ induces a $P_6$.

$\bullet$ $\{a_1,a_2\}$ is complete to $F_1$: Suppose that $f\in F_1$ such that $a_1f\notin E$, then $\{a_1,b_2,v_5,v_1,f,v_3\}$ induces a $P_6$.

$\bullet$ $\{a_1,a_2\}$ is complete to $F_3$: Suppose that $f\in F_3$ such that $a_1f\notin E$, then $\{a_1,b_2,v_5,f,v_3,b_1\}$ induces a $P_6$.

Then $a_1$ and $a_2$ are comparable. This completes the proof of \autoref{lem:H4free}.
\end{proof}

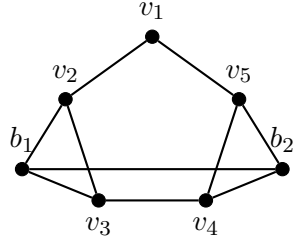
\begin{figure}[h!]
\centering
\begin{tikzpicture}[scale=0.6]
\tikzstyle{vertex}=[circle, draw, fill=black, inner sep=1pt, minimum size=5pt]
    \node[vertex, label=${v_1}$](1) at (0,2) {};
    \node[vertex, label=${v_2}$](2) at (-{2*cos(18)},{2*sin(18}) {};
    \node[vertex, label=below:${v_3}$](3) at (-{2*sin(36)},-{2*cos(36}) {};
    \node[vertex, label=below:${v_4}$](4) at ({2*sin(36)},-{2*cos(36}) {};
    \node[vertex, label=${v_5}$](5) at ({2*cos(18)},{2*sin(18}) {};
    \node[vertex, label=${b_1}$](6) at (-{3*cos(18)},-{3*sin(18}) {};
    \node[vertex, label=${b_2}$](7) at ({3*cos(18)},-{3*sin(18}) {};

    \Edge(1)(2)
    \Edge(2)(3)
    \Edge(3)(4)
    \Edge(4)(5)
    \Edge(1)(5)
    \Edge(2)(6)
    \Edge(3)(6)
    \Edge(4)(7)
    \Edge(5)(7)
    \Edge(6)(7)

\end{tikzpicture}
\caption{The graph $D_2$.}\label{fig:D2}
\end{figure}

\begin{lemma}\label{lem:H5free}
Every 6-vertex-critical ($P_6$,diamond)-free graph with clique number 3 is $D_2$-free. (See \autoref{fig:D2} for $D_2$.)
\end{lemma}

\begin{proof}
Let $G=(V,E)$ be a 6-vertex-critical ($P_6$,diamond)-free graph with clique number 3 containing a $D_2$. In the proof of this lemma, we follow the partitioning of $V$ in \autoref{sec:imperfect}. Let the $C_5$ in the $D_2$ be $Q$, and the other two vertices $b_1\in B_{2,3}$, $b_2\in B_{4,5}$ such that $b_1b_2\in E$. Then $B=\{b_1,b_2\}$.
By the properties from \autoref{sec:imperfect} and $\omega(G)=3$, $A_1$, $A_3$ and $A_4$ are stable, and $F_2\cup F_5$ is empty.

\begin{claim}\label{clm:D2free-1}
One of $A_1$ and $A_3\cup A_4$ is empty.
\end{claim}
\begin{proof}
Suppose that $a_1\in A_1$, $a_2\in A_3$, then $\{a_1,a_2,b_1,b_2\}$ induces a diamond.
\end{proof}

\begin{claim}\label{clm:D2free-2}
$B_{2,3}$ and $C_{4,1}$, $B_{4,5}$ and $C_{1,3}$ are complete.
\end{claim}
\begin{proof}
Suppose that $b\in B_{2,3}$ and $c\in C_{4,1}$ are nonadjacent, then $\{v_3,b_1,b_2,v_5,v_1,c\}$ induces a $P_6$.
\end{proof}

\begin{claim}\label{clm:D2free-3}
$A_2$ and $C_{4,1}$, $A_5$ and $C_{1,3}$ are complete.
\end{claim}
\begin{proof}
Suppose that $a\in A_2$ and $c\in C_{4,1}$ are nonadjacent, then $\{a,v_2,b_1,c,v_4,v_5\}$ induces a $P_6$.
\end{proof}

\begin{claim}\label{clm:D2free-4}
If $A_1$ is nonempty, then $A_2$ and $B_{4,5}$, $A_5$ and $B_{2,3}$ are complete.
\end{claim}
\begin{proof}
By symmetry let $a_1\in A_2$ such that $a_1b_2\notin E$. If $a_2\in A_1$, then $\{a_1,v_2,v_1,a_2,b_2,v_4\}$ induces a $P_6$.
\end{proof}

\begin{claim}\label{clm:D2free-5}
$A_2$ and $A_5$ are stable.
\end{claim}
\begin{proof}
By symmetry suppose that $A_2$ is nonstable. Then $A_4\cup A_5$ is empty. If $a\in A_1$, then $A_2\cup\{v_2,b_2\}$ contains a diamond by \autoref{clm:D2free-4}. If $c\in C_{4,1}$, then $A_2\cup\{v_2,c\}$ contains a diamond by \autoref{clm:D2free-3}. So $A_1\cup C_{4,1}$ is empty. Since $b_1$ has two neighbours $v_2,v_3$ in $Q$, a neighbour $b_2$, and no neighbours in $A\cup C\cup F$, $d(b_1)\leq3$.
\end{proof}

\begin{claim}\label{clm:D2free-6}
$A_3$ and $C_{1,3}$, $A_4$ and $C_{4,1}$ are anti-complete.
\end{claim}
\begin{proof}
By symmetry suppose that $a\in A_3$, $c\in C_{1,3}$ such that $ac\in E$. By \autoref{clm:D2free-2}, $cb_2\in E$. Then $\{a,c,b_2,v_3\}$ induces a diamond.
\end{proof}

Now we give a 5-colouring of $G$ as follows.

Suppose that $V\setminus Z$ has a 5-colouring such that one colour of $1,\ldots,5$ is not used in $C_{1,3}\cup C_{4,1}\cup F$, and $F$ is coloured with at most 2 colours, then we can extend it to a 5-colouring of $G$ as follows. For each component $K$ of $Z$, if $K$ is trivial, colour it with the colour not used in $C_{1,3}\cup C_{4,1}\cup F$, since $Z$ and $C_{2,4}\cup C_{3,5}\cup C_{5,2}$ are anti-complete by~\ref{item:26}. If $K$ is nontrivial, then either $K\cup N_C(K)$ is a clique of size at most 3 if $K$ has a neighbour in $C$, or $K$ is 3-colourable if $K$ has no neighbours in $C$, then $K$ can be coloured with the three colours not used in $F$. Therefore, it remains to give a colouring of $V\setminus Z$, so that one colour of $1,\ldots,5$ is not used in $C_{1,3}\cup C_{4,1}\cup F$, and $F$ is coloured with at most 2 colours.

\case{1} $A_3\cup A_4$ is empty.

\case{1.1} One of $A_2$ and $A_5$ is empty.

By symmetry assume that $A_5$ is empty.

$\bullet$ Colour $v_1,\ldots,v_5$ with colours $2,3,1,3,1$ in order.

$\bullet$ Colour $A_1$ and $A_2$ with colour $1$.

$\bullet$ Colour $B_{2,3}$ with colour $5$, $B_{4,5}$ with colour $2$.

$\bullet$ Colour $C_{1,3},C_{3,5},C_{4,1},C_{5,2}$ with colours $3,5,4,2$ in order. For each vertex in $C_{2,4}$, colour it with colour 1 if it has a neighbour in $C_{5,2}$, otherwise colour it with colour 2. If there is a vertex in $C_{2,4}$ with colour 1, then $A_1$ is empty by~\ref{item:23}. If $a\in A_2$ and $c_1\in C_{2,4}$ with colour 1 are adjacent, then there is a vertex $c_2\in C_{5,2}$ adjacent to $c_1$, and so $\{a,v_2,c_1,c_2\}$ induces a diamond or $K_4$. So $C$ and $A$ do not conflict.

$\bullet$ Colour $F_1,F_3,F_4$ with colours $5,5,4$ in order.

\case{1.2} Both $A_2$ and $A_5$ are nonempty.

Each vertex in $C_{2,4}\cup C_{3,5}$ is complete to one of $A_2$ and $A_5$, and anti-complete to the other. The proof can be reviewed in Case 1.3.2.3 of \autoref{thm:imperfect}, \autoref{sec:bound}.

$\bullet$ Colour $v_1,\ldots,v_5$ with colours $2,3,4,3,5$ in order.

$\bullet$ Colour $A_1,A_2,A_5$ with colours $4,1,2$ in order.

$\bullet$ Colour $B_{2,3}$ with colour $1$, $B_{4,5}$ with colour $2$.

$\bullet$ Colour $C_{1,3}$ with colour 3, $C_{4,1}$ with colour 5. Colour $C_{5,2}$ with colour 2 if $A_1$ is nonempty, otherwise colour $C_{5,2}$ with colour 4. Colour the vertices in $C_{2,4}\cup C_{3,5}$ which are adjacent to $A_2$ with colour $2$, and those adjacent to $A_5$ with colour $1$. Note that $C_{5,2}$ is anti-complete to $A_5$ by~\ref{item:25}, and anti-complete to $C_{2,4}\cup C_{5,2}$ when $A_1$ is nonempty by~\ref{item:23}.

$\bullet$ Colour $F_1,F_3,F_4$ with colours $5,3,5$ in order.

\case{2} $A_3\cup A_4$ is nonempty.

Then $A_1$ is empty by \autoref{clm:D2free-1}. By symmetry assume that $A_3$ is nonempty.

$\bullet$ Colour $v_1,\ldots,v_5$ with colours $1,4,3,1,5$ in order.

$\bullet$ Colour $A_2,A_3,A_4,A_5$ with colours $1,1,2,3$ in order.

$\bullet$ Colour $B_{2,3}$ with colour $1$, $B_{4,5}$ with colour $2$.

$\bullet$ Colour $C_{1,3},C_{2,4},C_{3,5},C_{4,1},C_{5,2}$ with colours $5,2,4,2,3$ in order. Note that $C_{2,4}$ and $C_{4,1}$ are anti-complete by~\ref{item:23}, $C_{2,4}$ and $A_4$ are anti-complete by~\ref{item:25}, $C_{4,1}$ and $A_4$ are anti-complete by \autoref{clm:D2free-6}.

$\bullet$ Colour $F_1,F_3,F_4$ with colours $5,4,5$ in order.

So $G$ is 5-colourable, a contradiction.
\end{proof}

\begin{lemma}\label{lem:H1free}
Every 6-vertex-critical ($P_6$,diamond)-free graph with clique number 3 is $S_1$-free. (See \autoref{fig:Sn} for $S_1$.)
\end{lemma}

\begin{proof}
Let $G=(V,E)$ be a 6-vertex-critical ($P_6$,diamond)-free graph with clique number 3 containing an $S_1$. In the proof of this lemma, we follow the partitioning of $V$ from \autoref{sec:imperfect}. Let the $C_5$ in the $S_1$ be $Q$, and the other vertex $b$ in $B_{2,3}$. By \autoref{lem:H4free} and \autoref{lem:H5free}, $G$ is ($D_1$,$D_2$)-free, and so $B=\{b\}$. Since $\omega(G)=3$, $F_5$ is empty.

\begin{claim}\label{clm:S1free-1}
$C_{5,2}\cup C_{3,5}$ and $A_5$ are anti-complete.
\end{claim}
\begin{proof}
By symmetry suppose that $c\in C_{5,2}$ and $a\in A_5$ are adjacent. Then $\{c,v_2,v_3,v_4,v_5,b,a\}$ induces a $D_1$ or $D_2$.
\end{proof}

\begin{claim}\label{clm:S1free-2}
$A_1\cup A_4$ and $C_{4,1}$ are anti-complete.
\end{claim}
\begin{proof}
By symmetry suppose that $a\in A_1$ and $c\in C_{4,1}$ are adjacent. By~\ref{item:7}, $ab\in E$. Then either $\{a,b,c,v_1\}$ induces a diamond, or $\{v_1,v_2,v_3,v_4,c,b,a\}$ induces a $D_2$, depending on whether $c$ and $b$ are adjacent.
\end{proof}

Now we give a 5-colouring of $G$ as follows.

Suppose that one colour of $1,\ldots,5$ is not used in $C_{1,3}\cup C_{2,4}\cup C_{4,1}\cup F$, and $F$ is coloured with at most 2 colours. For each component $K$ of $Z$, if $K$ is trivial, colour it with the colour not used in $C_{1,3}\cup C_{2,4}\cup C_{4,1}\cup F$, since $Z$ and $C_{3,5}\cup C_{5,2}$ are anti-complete by~\ref{item:26}. If $K$ is nontrivial, then either $K\cup N_C(K)$ is a clique of size at most 3 if $K$ has a neighbour in $C$, or $K$ is 3-colourable if $K$ has no neighbours in $C$, then $K$ can be coloured with the three colours not used in $F$. Now we give a colouring of $V\setminus Z$, so that one colour of $1,\ldots,5$ is not used in $C_{1,3}\cup C_{2,4}\cup C_{4,1}\cup F$, and $F$ is coloured with at most 2 colours. We consider four cases which are similar to the subcases of Case 2 in \autoref{sec:bound}.

\case{1} Both $A_2$ and $A_3$ are nonstable.

$\bullet$ Colour $v_1,\ldots,v_5$ with colours $1,3,1,3,2$ in order.

$\bullet$ Colour $A_2$ with colours in $\{1,2\}$ so that the trivial components are coloured with colour~$1$. Colour $A_3$ with colours in $\{2,3\}$ so that the trivial components are coloured with colour $2$.

$\bullet$ Colour $b$ with colour $2$.

$\bullet$ Colour $C_{2,4},C_{3,5},C_{4,1},C_{5,2}$ with colours $4,3,5,1$ in order. For each vertex in $C_{1,3}$, colour it with colour $2$ if it has a neighbour in $C_{3,5}$, otherwise colour it with colour $3$.

$\bullet$ Colour $F_1,F_2,F_3,F_4$ with colours $5,5,5,5$ in order.

\case{2} Exactly one of $A_2$ and $A_3$ is nonstable.

By symmetry assume that $A_2$ is nonstable. If $A_1$ is empty, then it reduces to Case 1. Now suppose that $A_1$ is nonempty.

$\bullet$ Colour $v_1,\ldots,v_5$ with colours $1,4,2,1,2$ in order.

$\bullet$ Colour $A_2$ with colours in $\{1,3\}$ so that the trivial components are coloured with colour~$1$. Colour $A_1$ with colour $2$, $A_3$ with colour $1$.

$\bullet$ Colour $b$ with colour $1$.

$\bullet$ Colour $C_{1,3},C_{2,4},C_{3,5},C_{4,1},C_{5,2}$ with colours $4,3,5,2,3$ in order. By \autoref{clm:S1free-2}, $A_1$ and $C_{4,1}$ do not conflict.

$\bullet$ Colour $F_1,F_2,F_3,F_4$ with colours $5,5,5,3$ in order.

\case{3} $A_5$ is nonstable.

Now we prove that $C_{1,3}\cup C_{2,4}$ is empty. By symmetry suppose that $c\in C_{1,3}$. Since $A_5$ is nonstable, let $a_1,a_2\in A_5$ such that $a_1a_2\in E$. Then at least one of $a_1$ and $a_2$ is nonadjacent to $b_1$, assume that $a_1b\notin E$. Then $ca_1\in E$, or else $\{b,v_3,c,v_1,v_5,a_1\}$ induces a $P_6$. So $a_2b\in E$, or else $ca_2\in E$ and then $\{c,a_1,c_2,v_5\}$ induces a diamond. Then $\{b,v_2,v_1,v_5,a_2,a_1,v_3\}$ induces a $D_1$. So $C_{1,3}\cup C_{2,4}$ is empty.

$\bullet$ Colour $v_1,\ldots,v_5$ with colours $4,2,4,1,5$ in order.

$\bullet$ Colour $A_5$ with colours in $\{1,2\}$ so that the trivial components are coloured with colour~$1$. Colour $A_1$ with colour $1$, $A_4$ with colour $2$.

$\bullet$ Colour $b$ with colour $3$.

$\bullet$ Colour $C_{3,5},C_{4,1},C_{5,2}$ with colours $3,5,4$ in order.

$\bullet$ Colour $F_1,F_2,F_3,F_4$ with colours $5,3,3,5$ in order.

\case{4} None of $A_1,\ldots,A_5$ is nonstable.

\case{4.1} $A_1\cup A_4$ is nonempty.

By symmetry assume that $A_4$ is nonempty.

$\bullet$ Colour $v_1,\ldots,v_5$ with colours $2,3,1,2,3$ in order.

$\bullet$ Colour $A_1,\ldots,A_5$ with colours $1,2,2,3,1$ in order.

$\bullet$ Colour $b$ with colour $4$.

$\bullet$ Colour $C_{1,3},C_{2,4},C_{3,5},C_{4,1}$ with colours $5,4,5,3$ in order. If $A_1$ is empty, then colour $C_{5,2}$ with colour 1, otherwise colour $C_{5,2}$ with colour 4. By~\ref{item:23}, $C_{2,4}$ and $C_{5,2}$, $C_{1,3}$ and $C_{3,5}$ are anti-complete. By \autoref{clm:S1free-1}, $C_{5,2}$ and $A_5$ do not conflict. By \autoref{clm:S1free-2}, $A_4$ and $C_{4,1}$ do not conflict.

$\bullet$ Colour $F_1,F_2,F_3,F_4$ with colours $5,4,5,4$ in order.

\case{4.2} $A_1\cup A_4$ is empty.

$\bullet$ Colour $v_1,\ldots,v_5$ with colours $1,2,1,2,4$ in order.

$\bullet$ Colour $A_2,A_3,A_5$ with colours $1,2,3$ in order.

$\bullet$ Colour $b$ with colour $5$.

$\bullet$ Colour $C_{1,3},C_{2,4},C_{3,5},C_{5,2}$ with colours $4,5,2,1$ in order. For each vertex in $C_{4,1}$, colour it with colour $3$ if it has a neighbour in $C_{1,3}$, otherwise colour it with colour $4$. Suppose that $c_1\in C_{4,1}$ with colour $3$ is adjacent to $a\in A_5$, then there is a vertex $c_2\in C_{1,3}$ adjacent to $c_1$, then either $\{v_1,c_1,c_2,a\}$ induces a diamond, or $\{v_2,v_3,c_2,c_1,a,v_5\}$ induces a $P_6$, depending on whether $a$ and $c_2$ are adjacent. So $C_{4,1}$ and $A_5$ do not conflict.

$\bullet$ Colour $F_1,F_2,F_3,F_4$ with colours $5,5,3,5$ in order.

So $G$ is 5-colourable, a contradiction.
\end{proof}

\begin{lemma}\label{lem:domi}
Every 6-vertex-critical ($P_6$,diamond)-free graph with clique number $\omega$ ($\omega=3,4,5$) is ($K_{\omega}+K_1$)-free.
\end{lemma}

\begin{proof}
Let $G=(V,E)$ be a 6-vertex-critical ($P_6$,diamond)-free graph with clique number $\omega$. By \autoref{lem:Hfree-45} and \autoref{lem:H1free}, $G$ is $S_{\omega-2}$-free. Let $A=\{a_1,\ldots,a_{\omega}\}$ be a $K_{\omega}$ in $G$. Let $B_i=\{v\in V\setminus A: va_i\in E\}$ for $i=1,\ldots,\omega$, $B=\bigcup_{i=1}^{\omega}B_i$, $C=V\setminus(A\cup B)$. Since $A$ is an arbitrary $K_{\omega}$, we just need to prove that $C$ is empty.

Since $G$ is (diamond,$K_{\omega+1}$)-free and $\delta(G)\geq5$, we have some obvious properties of $B$:

$\bullet$ $B_1,\ldots,B_{\omega}$ are pairwise disjoint.

$\bullet$ Each component of $B_i$ is a clique of size at most $\omega-1$.

$\bullet$ Every $b\in B_i$ has at most one neighbour in each component of $B_j$ for $i\neq j$.

$\bullet$ Each $B_i$ contains at least $6-\omega$ vertices.

\begin{claim}\label{clm:domi-antiB}
Each vertex in $C$ is anti-complete to $B_i$ for some $i\in\{1,\ldots,\omega\}$.
\end{claim}
\begin{proof}
Let $c$ be an arbitrary vertex in $C$. Since every $B_i$ is nonempty, suppose that $\{b_1,\ldots,b_{\omega}\}\subseteq N_B(c)$ such that $b_j\in B_j$ for $j=1,\ldots,\omega$. Since $G$ is $S_{\omega-2}$-free, $b_1,\ldots,b_{\omega}$ are pairwise adjacent. Then $\{b_1,\ldots,b_{\omega},c\}$ induces a $K_{\omega+1}$.
\end{proof}


\begin{claim}\label{clm:domi-Cijfree}
No vertices in $C$ have a neighbour in $B_i$ and a neighbour $B_j$ for $i\neq j$.
\end{claim}
\begin{proof}
Suppose that $c_1\in C$ has two neighbours $b_1\in B_1$ and $b_2\in B_2$. Since $G$ is $S_{\omega-2}$-free, $b_1b_2\in E$. Moreover, for each neighbour $b$ of $c_1$ in $B$, $bb_1\in E$ if $b\notin B_1$, and $bb_2\in E$ if $b\notin B_2$. So $N_B(c_1)$ is connected, and forms a clique since $G$ is diamond-free. Since $G$ has no clique cutsets, $G-N_B(c_1)$ is connected. Assume that $c_1,c_2,\ldots,c_n,b_3,a_i$ $(n\geq2)$ is a shortest path between $c_1$ and $A$ in $G-N_B(c_1)$. Then $c_2,\ldots,c_n\in C$. If $n\geq3$, then $\{c_1,c_2,\ldots,c_n,b_3,a_i,a_j\}$ contains a $P_6$ for some $j$. So $n=2$.

Since $\{c_1,c_2,b_1,b_2\}$ cannot induce a diamond, either $c_2b_1,c_2b_2\in E$ or $c_2b_1,c_2b_2\notin E$. Assume that $b_1c_2,b_2c_2\in E$. By symmetry assume that $b_3\notin B_1$. Then $b_1b_3\in E$, and so $\{c_1,c_2,b_1,b_3\}$ induces a diamond. So $b_1c_2,b_2c_2\notin E$. Since $\{c_1,b_1,b_2,b_3\}$ cannot induce a diamond, $b_3$ is nonadjacent to at least one of $b_1$ and $b_2$. If $b_1b_3\notin E$, then $b_3\in B_1$, or else $\{b_1,c_1,c_2,b_3,a_i,a_j\}$ induces a $P_6$ for some $i,j$. Similarly, if $b_2b_3\notin E$, then $b_3\in B_2$. So by symmetry we may assume that $b_3\in B_1$, $b_1b_3\notin E$ and $b_2b_3\in E$.

By \autoref{clm:domi-antiB} we may assume that $b_4\in B_3$ is nonadjacent to $c_1$. If $c_2b_4\in E$, then $b_3b_4\in E$. Then either $\{c_2,b_2,b_3,b_4\}$ induces a diamond, or $\{b_4,c_2,c_1,b_2,a_2,a_1\}$ induces a $P_6$, depending on whether $b_2$ and $b_4$ are adjacent. So $c_2b_4\notin E$. Then $\{c_2,c_1,b_2,b_4,a_3,a_1\}$ or $\{c_2,c_1,b_2,a_2,a_3,b_4\}$ induces a $P_6$, depending on whether $b_2$ and $b_4$ are adjacent.
\end{proof}

Now suppose that $c_1\in C$ has a neighbour in $b_1\in B_1$, then $c_1$ has no neighbours in $B\setminus B_1$ by \autoref{clm:domi-Cijfree}. If $c_1$ has no neighbours in $C$, then $N_G(c_1)\subseteq N_G(a_1)$. So $c_1$ has a neighbour $c_2\in C$. This implies that $C$ has no trivial components. By \autoref{clm:domi-Cijfree} we may assume that $b_2\in B_2$ is nonadjacent to $c_2$. If $c_2b_1\notin E$, then $\{c_2,c_1,b_1,b_2,a_2,a_3\}$ or $\{c_2,c_1,b_1,a_1,a_2,b_2\}$ induces a $P_6$, depending on whether $b_1$ and $b_2$ are adjacent. So $N_B(c_2)\subseteq N_B(c_1)$. Moreover by symmetry and transitivity, every vertex in a component $K$ of $C$ has the same neighbourhood in $B$. Then either $N_B(K)$ is a clique cutset, or $K\cup N_B(K)$ contains a diamond. So $C$ is empty.
\end{proof}

\subsection{Proof of \autoref{thm:omega4}}\label{ssc:omega4}

Let $G=(V,E)$ be a 6-vertex-critical ($P_6$,diamond)-free graph with clique number 4. Let $A=\{a_1,a_2,a_3,a_4\}$ be a $K_4$ in $G$, $B_i=\{v\in V\setminus A: va_i\in E\}$ for $i=1,2,3,4$, $B=\bigcup_{i=1}^{4}B_i$. By \autoref{lem:domi}, $G$ is ($K_4+K_1$)-free. So $V=A\cup B$ and we may use the properties of $B$ mentioned in \autoref{lem:domi}.

\begin{claim}\label{clm:omega4-3bn2c}
$B\setminus B_i$ is not 2-colourable for $i=1,2,3,4$.
\end{claim}
\begin{proof}
By symmetry suppose that $B_2\cup B_3\cup B_4$ is 2-colourable. Then $G$ has a 5-colouring: colour $a_1$ with colour 1, $B_1\cup \{a_2,a_3,a_4\}$ with colours in $\{2,3,4\}$, and $B_2\cup B_3\cup B_4$ with colours in $\{1,5\}$.
\end{proof}

\begin{claim}\label{clm:omega4-2k3}
If $B_i$ contains a $K_3$, then $B_j$ contains no $K_3$ for $j\neq i$.
\end{claim}
\begin{proof}
Let $K=\{b_1,b_2,b_3\}$ be a $K_3$ in $B_1$, and $L=\{b_4,b_5,b_6\}$ be a $K_3$ in $B_2$. Let $b_7\in B_3$. Since $G$ is (diamond,$K_4+K_1$)-free, $b_7$ has exactly one neighbour in $K$ and exactly one neighbour in $L$, assume that $b_1b_7,b_4b_7\in E$. At least one of $b_5$ and $b_6$ is nonadjacent to $b_1$, assume that $b_5b_1\notin E$. One of $b_2$ and $b_3$ is nonadjacent to $b_5$, assume that $b_2b_5\notin E$. Then $\{b_5,a_2,a_3,b_7,b_1,b_2\}$ induces a $P_6$.
\end{proof}

\begin{claim}\label{clm:omega4-bnc5}
$B$ contains no $C_5$.
\end{claim}
\begin{proof}
We use some properties from \autoref{sec:imperfect} to prove this claim. Let $Q$=\{$v_1,v_2,v_3,v_4,v_5$\} induce a $C_5$ in $B$ with edges $v_iv_{i+1}$ for $i=1,\ldots,5$, all indices modulo 5. By~\ref{item:2} and~\ref{item:21}, $Q\subseteq B_i\cup B_j$. Assume that $v_1,v_3\in B_1, v_2,v_4,v_5\in B_2$.

Let $u$ be an arbitrary vertex in $B_3\cup B_4$. By symmetry let $u\in B_3$. If $u$ has no neighbours in $Q$, then $\{u,a_3,a_1,v_1,v_5,v_4\}$ induces a $P_6$. By~\ref{item:21}, $u$ is adjacent to two or three nonsequential vertices in $Q$. If $u$ has exactly two neighbours $v_i$ and $v_{i+2}$ in $Q$, then $\{a_4,a_3,u,v_i,v_{i+4},v_{i+3}\}$ induces a $P_6$. Since $v_4$ and $v_5$ are in the same component of $B_2$, $N_Q(u)\neq\{v_2,v_4,v_5\}$. So $N_Q(u)=\{v_1,v_2,v_4\}$ or $\{v_1,v_3,v_4\}$ or $\{v_1,v_3,v_5\}$ or $\{v_2,v_3,v_5\}$. Note that each $B_i$ contains at least two vertices. Then by~\ref{item:13}, $B_3\cup B_4$ is stable, and each of $B_3$ and $B_4$ contains exactly two vertices.

By \autoref{clm:omega4-3bn2c}, $B_1\cup B_3\cup B_4$ contains a $K_3$ or a $C_5$. If $B_1\cup B_3\cup B_4$ contains a $C_5$, then $B_2$ contains exactly two vertices, a contradiction. So $B_1\cup B_3\cup B_4$ contains a $K_3$. Since there are no edges between $B_3$ and $B_4$, $B_1$ contains a $K_3$. Then $B_2\cup B_3\cup B_4$ contains no $K_3$ by \autoref{clm:omega4-2k3} and so contains a $C_5$. Then $B_1$ contains exactly two vertices, a contradiction.
\end{proof}

\begin{claim}\label{clm:omega4-k3-bhask4}
For every component $K$ of $B_i$ of size 3, there is a $K_4$ in $B$ containing one vertex in each of $B_1,\ldots,B_4$, and containing one vertex in $K$.
\end{claim}
\begin{proof}
Assume that $K=\{b_1,b_2,b_3\}$ is a $K_3$ in $B_1$. By \autoref{clm:omega4-3bn2c}, \autoref{clm:omega4-2k3} and \autoref{clm:omega4-bnc5}, there are $b_4\in B_2, b_5\in B_3, b_6\in B_4$ such that $L=\{b_4,b_5,b_6\}$ is a $K_3$. Since $G$ is diamond-free, every vertex in $K$ cannot have exactly two neighbours in $L$. If some vertex in $K$ is complete to $L$, then we are done. So suppose that every vertex in $K$ has at most one neighbour in $L$. Suppose that $b_1b_4\in E$. By symmetry we may assume that $b_2b_5\notin E$. Then $\{b_2,b_1,b_4,b_5,a_3,a_4\}$ induces a $P_6$. So $K$ and $L$ are anti-complete, and then $K\cup\{a_1,b_4\}$ induces a $K_4+K_1$.
\end{proof}

\begin{claim}\label{clm:omega4-k2-bhask4}
If no $B_j$ contains a $K_3$ for $j=1,2,3,4$, then for every component $K$ of $B_i$ of size~2, there is a $K_4$ in $B$ containing one vertex in each of $B_1,\ldots,B_4$, and containing one vertex in $K$.
\end{claim}
\begin{proof}
Assume that no $B_j$ contains a $K_3$ for $j=1,2,3,4$, and $K=\{b_1,b_2\}$ is a $K_2$ in $B_1$. By \autoref{clm:omega4-3bn2c} and \autoref{clm:omega4-bnc5}, there are $b_3\in B_2, b_4\in B_3, b_5\in B_4$ such that $L=\{b_3,b_4,b_5\}$ is a $K_3$. Since $G$ is diamond-free, every vertex in $K$ cannot have exactly two neighbours in $L$. If some vertex in $K$ is complete to $L$, then we are done. So suppose that every vertex in $K$ has at most one neighbour in $L$. By a similar proof as for the proof of \autoref{clm:omega4-k3-bhask4} we have that $K$ and $L$ are anti-complete.

Since $A$ is not a clique cutset, $B$ is connected. Assume that $b_1,u_1,\ldots,u_n,b_3$ is a shortest path between $K$ and $L$ in $B$. Since $K$ is a component of $B_1$, $u_1\notin B_1$, and so $v_2u_1\notin E$. Note that $u_n$ is either complete to $L$, or nonadjacent to $b_4,b_5$ since $G$ is diamond-free. If $u_n$ is complete to $L$, then $L\cup\{u_n,b_2\}$ induces a $K_4+K_1$. So $u_nb_4,u_nb_5\notin E$. If $n\geq2$, then $\{b_2,b_1,u_1,\ldots,u_n,b_3,b_4\}$ contains a $P_6$. So $n=1$. If $u_1\in B_2$, then $\{b_2,b_1,u_1,b_3,b_4,a_3\}$ induces a $P_6$. If $u_1\notin B_2$, by symmetry assume that $u_1\notin B_3$, then $\{b_2,b_1,u_1,b_3,a_2,a_3\}$ induces a $P_6$.
\end{proof}

If every $B_i$ is stable for $i=1,2,3,4$, then $G$ is 4-colourable. So one of $B_1,\ldots,B_4$ contains a component of size 3 or 2. By \autoref{clm:omega4-k3-bhask4} and \autoref{clm:omega4-k2-bhask4}, we may assume that $A'=\{b_1,b_2,b_3,b_4\}$ is a $K_4$ in $B$ such that $b_i\in B_i$ for $i=1,2,3,4$, and there is a vertex $v_1$ adjacent to $a_1$ and $b_1$.

Since $G$ is ($K_4+K_1$,diamond)-free, every vertex in $V\setminus(A\cup A')$ is adjacent to exactly one vertex in $A$ and exactly one vertex in $A'$. Let $v_2$ be a neighbour of $v_1$ in $V\setminus(A\cup A')$. Because $G$ is diamond-free, $v_2$ is adjacent to $a_1$ if and only if it is adjacent to $b_1$. By $\omega(G)=4$, $v_1$ has at most one neighbour in $V\setminus (A\cup A')$ which is adjacent to $a_1$ and $b_1$. Now assume that $v_2$ is nonadjacent to $a_1$. There are two cases up to symmetry: $v_2a_2,v_2b_2\in E$ or $v_2a_2,v_2b_3\in E$. In either case, $\{b_4,b_1,v_1,v_2,a_2,a_3\}$ induces a $P_6$. So $v_1$ has no neighbours in $V\setminus(A\cup A')$ which are nonadjacent to $a_1$. Then $d(v_1)\leq3$, a contradiction. So there are no 6-vertex-critical ($P_6$,diamond)-free graphs with clique number 4.

\subsection{Proof of \autoref{thm:omega5}}\label{ssc:omega5}

Let $G=(V,E)$ be a 6-vertex-critical ($P_6$,diamond)-free graph with clique number 5. Let $A=\{a_1,a_2,a_3,a_4,a_5\}$ be a $K_5$ in $G$, $B_i=\{v\in V\setminus A: va_i\in E\}$ for $i=1,2,3,4,5$, $B=\bigcup_{i=1}^{5}B_i$. By \autoref{lem:domi}, $G$ is ($K_5+K_1$)-free. So $V=A \cup B$ and we may use the properties of $B$ mentioned in \autoref{lem:domi}.

\begin{claim}\label{clm:omega5-2k3}
If $B_i$ contains a $K_3$, then $B_j$ contains no $K_3$ for $j\neq i$.
\end{claim}
\begin{proof}
Let $K=\{b_1,b_2,b_3\}$ be a $K_3$ in $B_1$, and $L=\{b_4,b_5,b_6\}$ be a $K_3$ in $B_2$. Now we prove that $K\cup L$ and $B\setminus(B_1\cup B_2)$ are anti-complete. Suppose that $b_7\in B_3$ is adjacent to $b_1$, then $b_7b_2,b_7b_3\notin E$. Since each of $b_1$ and $b_7$ has at most one neighbour in $L$, assume that $b_4b_1,b_4b_7\notin E$. We may assume that $b_2b_4\notin E$ by symmetry. Then $\{b_4,a_2,a_3,b_7,b_1,b_2\}$ induces a $P_6$. So $K\cup L$ and $B\setminus(B_1\cup B_2)$ are anti-complete.

Since $A$ is not a clique cutset, $B$ is connected. Assume that $b_1,u_1,\ldots,u_n,b_8$ is a shortest path between $K\cup L$ and $B\setminus(B_1\cup B_2)$ in $B$. By symmetry assume that $b_8\in B_3$. Then $u_1,\ldots,u_n\in B_1\cup B_2$. If $b_2u_1\notin E$, then $\{b_2,b_1,u_1,\ldots,u_n,b_8,a_3,a_4\}$ contains a $P_6$. So $b_2u_1\in E$, and then $K\cup\{u_1\}$ is a $K_4$ in $B_1$. Then $n=1$, or else $K\cup\{u_1,a_1,b_8\}$ induces a $K_5+K_1$. Since each of $b_1$ and $u_1$ has at most one neighbour in $L$, assume that $b_4b_1,b_4u_1\notin E$. Then $\{b_1,u_1,b_8,a_3,a_2,b_4\}$ induces a $P_6$.
\end{proof}

\begin{claim}\label{clm:omega5-k4inbi}
Every $B_i$ is $K_4$-free.
\end{claim}
\begin{proof}
Suppose that $K=\{b_1,b_2,b_3,b_4\}$ is a $K_4$ in $B_1$. If $B\setminus B_1$ is stable, then $G$ is 5-colourable. So let $L=\{b_5,b_6\}$ be a $K_2$ in $B\setminus B_1$, then each of $b_5$ and $b_6$ has exactly one neighbour $K$.

First we assume that $b_5$ and $b_6$ are in the same $B_i$, say $B_2$. Then their neighbours in $K$ are different. Assume that $b_1b_5,b_2b_6\in E$. Suppose that $b_5$ has a neighbour $b_7\in B_3$. Then $b_6b_7\notin E$. At least one of $b_3$ and $b_4$ is nonadjacent to $b_7$, assume that $b_3b_7\notin E$. Then $\{b_6,b_5,b_7,a_3,a_1,b_3\}$ induces a $P_6$. So $b_5$ has no neighbours in $B_3\cup B_4\cup B_5$. By \autoref{clm:omega5-2k3}, since $B_1$ contains a $K_4$, $b_5$ has only one neighbour $b_6$ in $B_2$. By $d(b_5)\geq5$, $b_5$ has a neighbour $b_8\in B_1\setminus K$. Then $b_8b_6\notin E$. By $d(b_3)\geq5$, $b_3$ has a neighbour $b_9\in B\setminus B_1$. Then $b_5b_9,b_6b_9\notin E$. If $b_9b_8\notin E$, then $\{b_8,b_5,b_6,b_2,b_3,b_9\}$ induces a $P_6$. If $b_9b_8\in E$, then $\{b_8,b_9,a_2,b_6,b_2,b_4\}$ or $\{b_8,b_9,b_3,b_2,b_6,a_2\}$ induces a $P_6$, depending on whether $b_9\in B_2$. So $b_5$ and $b_6$ cannot be in the same $B_i$.

Now we assume that $b_5\in B_2$, $b_6\in B_3$ by symmetry. Let $b_1b_5\in E$ and $b_2b_6\notin E$. If $b_1b_6\notin E$, then $\{b_2,b_1,b_5,b_6,a_3,a_4\}$ induces a $P_6$. So $b_1b_6\in E$. Note that $K\cup\{a_1\}$ is a $K_5$, while $A\setminus\{a_1\}$ is a $K_4$ in the neighbourhood of $a_1$ and $L$ is a $K_2$ in the neighbourhood of $b_1$. Then we can get a contradiction by symmetry.
\end{proof}

\begin{claim}\label{clm:omega5-bnc5}
$B$ contains no $C_5$.
\end{claim}
\begin{proof}
Let $Q$=\{$v_1,v_2,v_3,v_4,v_5$\} induce a $C_5$ in $B$ with edges $v_iv_{i+1}$ for $i=1,\ldots,5$, all indices modulo 5. By the properties from \autoref{sec:imperfect}, we may assume that $v_1,v_3\in B_1$, $v_2,v_4,v_5\in B_2$. Then for any $u\in B_3\cup B_4\cup B_5$, $N_Q(u)=\{v_1,v_2,v_4\}$ or $\{v_1,v_3,v_4\}$ or $\{v_1,v_3,v_5\}$ or $\{v_2,v_3,v_5\}$. And then $B_3\cup B_4\cup B_5$ is stable and contains at most 4 vertices whose neighbourhoods in $Q$ are pairwise different. Let $b_1\in B_3$, $b_2\in B_4$. Note that there is an induced $P_4=\{b_1,v_i,v_j,b_2\}$. Then $\{a_5,a_3,b_1,v_i,v_j,b_2\}$ induces a $P_6$.
\end{proof}

\begin{claim}\label{clm:omega5-3bn2c}
$B\setminus (B_i\cup B_j)$ is not 2-colourable for $i\neq j$.
\end{claim}
\begin{proof}
Suppose that $B_3\cup B_4\cup B_5$ is 2-colourable. Then $B_1\cup B_2$ is not 3-colourable, or else $G$ is 5-colourable. Since $B_1\cup B_2$ contains no $K_4$ by \autoref{clm:omega5-k4inbi}, $B_1\cup B_2$ is imperfect. Then $B_1\cup B_2$ contains a $C_5$, contradicting with \autoref{clm:omega5-bnc5}.
\end{proof}

Since $G$ is not 5-colourable, at least one of $B_1,\ldots,B_5$ is nonstable. By \autoref{clm:omega5-k4inbi}, the largest component among $B_1,\ldots,B_5$ is of size 2 or 3.

\begin{claim}\label{clm:omega5-k3-bhask4}
Let $K$ in $B_i$ be a largest component among $B_1,\ldots,B_5$. Then there is a $K_4$ containing one vertex in each of $K,B_{j_1},B_{j_2},B_{j_3}$ for every combination of three sets $B_{j_1},B_{j_2},B_{j_3}$ in $\{B_1,\ldots,B_5\}\setminus\{B_i\}$.
\end{claim}
\begin{proof}
Let $K=\{b_1,b_2\}$ or $\{b_1,b_2,b_3\}$ in $B_1$. By \autoref{clm:omega5-2k3}, \autoref{clm:omega5-bnc5} and \autoref{clm:omega5-3bn2c}, there are $b_4\in B_2, b_5\in B_3, b_6\in B_4$ such that $L=\{b_4,b_5,b_6\}$ is a $K_3$. If some vertex in $K$ is complete to $L$, then we are done. So suppose that every vertex in $K$ has at most one neighbour in $L$. Suppose that $b_1b_4\in E$. By symmetry we may assume that $b_2b_5\notin E$. Then $\{b_2,b_1,b_4,b_5,a_3,a_4\}$ induces a $P_6$. So $K$ and $L$ are anti-complete.

Since $B$ is connected, assume that $\{b_1,u_1,\ldots,u_n,b_4\}$ is a shortest path between $K$ and $L$ in $B$. Note that $u_1\notin B_1$, so $b_2u_1\notin E$. If $n\geq3$, then $\{b_2,b_1,u_1,\ldots,u_n,b_4\}$ contains a $P_6$. If $n=2$, then $u_2$ is complete to $L$, for otherwise $\{b_2,b_1,u_1,u_2,b_4,b_5\}$ induces a $P_6$. By symmetry we may assume that $u_1,u_2\notin B_2$. Then $\{b_2,b_1,u_1,u_2,b_4,a_2\}$ induces a $P_6$. So $n=1$. If $u_1\in B_2$, then $u_1b_5\notin E$, or else $\{u_1,b_4,b_5,a_2\}$ induces a diamond. Then $\{b_2,b_1,u_1,b_4,b_5,a_3\}$ induces a $P_6$. If $u_1\notin B_2$, by symmetry assume that $u_1\notin B_3$, then $\{b_2,b_1,u_1,b_4,a_2,a_3\}$ induces a $P_6$.
\end{proof}

\begin{claim}\label{clm:omega5-bhask5}
Let $K$ be a largest component among $B_1,\ldots,B_5$, then there is a $K_5$ in $B$ containing one vertex in each of $B_1,\ldots,B_5$, and containing one vertex in $K$.

\end{claim}
\begin{proof}
Let $K=\{b_1,b_2\}\ or\ \{b_1,b_2,b_3\}$ in $B_1$. By \autoref{clm:omega5-k3-bhask4}, there are $b_4\in B_2$, $b_5\in B_3$, $b_6\in B_4$ such that $\{b_1,b_4,b_5,b_6\}$ is a $K_4$, and there is a vertex $b_7\in B_5$ which has a neighbour in $K$. If $b_7$ is complete to $\{b_4,b_5,b_6\}$, then $b_7b_1\in E$ and so we are done. Suppose that $b_7$ has at most one neighbour in $\{b_4,b_5,b_6\}$. By symmetry assume that $b_7b_4,b_7b_5\notin E$.

If $b_7$ is adjacent to $b_2$ or $b_3$, by symmetry assume that $b_7b_2\in E$, then $\{b_7,b_2,b_1,b_4,a_2,a_3\}$ induces a $P_6$. So $b_7$ is adjacent to $b_1$. Then $b_7b_6\notin E$, or else $\{b_1,b_7,b_6,b_5\}$ induces a diamond. Then there are $b_8\in B_2$ and $b_9\in B_3$ such that $\{b_1,b_7,b_8,b_9\}$ is a $K_4$. Since $\{b_6,b_7,b_8,b_9\}$ cannot induce a diamond, at least one of $b_8$ and $b_9$ is nonadjacent to $b_6$, by symmetry assume that $b_8b_6\notin E$. Since $\{b_8,b_4,b_5,b_6\}$ cannot induce a diamond, at least one of $b_4$ and $b_5$ is nonadjacent to $b_8$. If $b_8b_4\notin E$, then $\{b_8,b_7,a_5,a_4,b_6,b_4\}$. If $b_8b_5\notin E$, then $\{b_8,b_7,a_5,a_4,b_6,b_5\}$ induces a $P_6$.
\end{proof}

Now we may assume that $A'=\{b_1,b_2,b_3,b_4,b_5\}$ is a $K_5$ in $B$ such that $b_i\in B_i$ for $i=1,2,3,4,5$, and there is a vertex $v_1$ adjacent to $a_1$ and $b_1$. Similarly as in the proof of \autoref{thm:omega4} we consider the neighbourhood of $v_1$. Then $d(v_1)\leq4$, a contradiction. So there are no 6-vertex-critical ($P_6$,diamond)-free graphs with clique number 5.

\section{Conclusion}\label{sec:conclude}

In this paper, we improved the $\chi$-bound of ($P_6$, diamond)-free graphs from $\chi(G)\leq\omega(G)+3$~\cite{CHM18} to $\chi(G)\leq \max\{6,\omega(G)\}$. Moreover, we proved that the chromatic number of graphs in the class of ($P_6$, diamond)-free graphs can be calculated in polynomial time. We suspect that similar results can be obtained for other hereditary graph families: if a hereditary graph family has a $\chi$-bound in the form of $\chi(G)\leq\omega(G)+C$ where $C$ is a constant, then it may be possible to improve the bound to $\max\{C',\omega(G)\}$ where $C'$ is a constant. If that is the case, it may also be possible to compute the chromatic number in polynomial time for this family of graphs. However, this is not always possible since determining the chromatic number is NP-hard for the hereditary family of line graphs, which has a $\chi$-bound in the form of $\chi(G)\leq\omega(G)+C$. So it would be interesting to find other hereditary graph families for which the chromatic number can be determined in polynomial time in this way.


\section*{Appendix}

\subsection*{A Computer-Free Proof  of a Weaker Version of \autoref{thm:omega3}}\label{ssc:omega3_manual}

The proof of \autoref{thm:omega3} uses computational methods. In this section we give a computer-free proof of a weaker version of \autoref{thm:omega3}. This weaker theorem still suffices to give a complete computer-free proof of \autoref{thm:poly} and \autoref{thm:K6free}.

\begin{theorem}\label{thm:omega3_manual}
There are finitely many 6-vertex-critical ($P_6$,diamond)-free graphs with clique number 3.
\end{theorem}

\begin{proof}[Proof of \autoref{thm:omega3_manual}]
Let $G=(V,E)$ be a 6-vertex-critical ($P_6$,diamond)-free graph with clique number 3. Let $A=\{a_1,a_2,a_3\}$ be a $K_3$ in $G$, $B_i=\{v\in V\setminus A: va_i\in E\}$ for $i=1,2,3$, $B=\bigcup_{i=1}^{3}B_i$. By \autoref{lem:domi}, $G$ is ($K_3+K_1$)-free. So $V=A\cup B$ and we may use the properties of $B$ mentioned in \autoref{lem:domi}.

\begin{claim}\label{clm:omega3-11vertex}
$B\setminus B_i$ is $K_3$-free and not 3-colourable for $i=1,2,3$.
\end{claim}
\begin{proof}
By symmetry assume that $i=3$. Since $G$ is $K_4$-free, each $B_j$ cannot contain a $K_3$. If $\{b_1,b_2,b_3\}$ is a $K_3$ in $B_1\cup B_2$, by symmetry assume that $b_1,b_2\in B_1$, $b_3\in B_2$, then $\{b_1,b_2,b_3,a_1\}$ induces a diamond. So $B_1\cup B_2$ is $K_3$-free. If $B_1\cup B_2$ is 3-colourable, then we have a 5-colouring for $G$: colour $a_1,a_2,a_3$ with colours $1,2,3$ in order, colour $B_3$ with colours in $\{1,2\}$ and $B_1\cup B_2$ with colours in $\{3,4,5\}$. So $B_1\cup B_2$ is not 3-colourable.
\end{proof}

\begin{claim}\label{clm:omega3-16vertex}
$B\setminus B_i$ is an induced subgraph of the Clebsch graph for $i=1,2,3$.
\end{claim}
\begin{proof}
By symmetry let $i=3$. By \autoref{thm:P6C3} and \autoref{clm:omega3-11vertex}, every $B_1\cup B_2$ contains the Gr{\"o}tzsch graph. Let $K$ be a component of $B_1\cup B_2$ such that $K$ contains the Gr{\"o}tzsch graph. Let $L$ be a graph obtained from $K$ by successively deleting the vertex with the smaller neighbourhood of every pair of comparable vertices (if the neighbourhoods are equal, pick any of the two vertices). 
Let $u_L$ be the last deleted vertex, if it exists. Then $L$ contains the Gr{\"o}tzsch graph, and is an induced subgraph of the Clebsch graph by \autoref{thm:P6C3}. Let $L=\{v_0,\ldots,v_{10},\ldots\}$, whose indices correspond to \autoref{fig:Clebsch}. Note that each of $v_{11},\ldots,v_{15}$ may either exist or not exist.

Now we prove that every vertex in $(B_1\cup B_2)\setminus L$ has a neighbour in $L$. Suppose that $u\in (B_1\cup B_2)\setminus L$ has no neighbours in $L$. Since $B_1$ and $B_2$ are $P_3$-free, every induced $C_5$ in $B_1\cup B_2$ has the property that two nonadjacent vertices of the $C_5$ belong to one of $B_1$ and $B_2$, and the other three vertices belong to the other of $B_1$ and $B_2$. Now we consider $\{v_6,v_7,v_8,v_9,v_{10}\}$, by symmetry we may assume that $v_6,v_8\in B_1$, $v_7,v_9,v_{10}\in B_2$. Then $v_5\in B_1$. Then $\{u,v_9,v_{10},a_2\}$ or $\{u,v_5,v_6,a_1\}$ induces a $K_3+K_1$, depending on whether $u$ belongs to $B_1$. So every vertex in $(B_1\cup B_2)\setminus L$ has a neighbour in $L$. This implies that $B_1\cup B_2$ is connected.

Now assume that $u_L$ exists, and by symmetry assume that $u_L\in B_1$. Since $u_L$ is the last deleted vertex when we obtain $L$, there is a vertex $u_1\in L$ such that $N_L(u_L)\subseteq N_L(u_1)$. There is a vertex $u_2\in N_L(u_1)$ such that $u_Lu_2\in E$. Note that every $K_2$ in $L$ is in a $C_5$ in $L$. Let $\{u_1,u_2,u_3,u_4,u_5\}$ be an arbitrary $C_5$ in $L$ containing $u_1$ and $u_2$, such that $v_iv_{i+1}\in E$ for $i=1,\ldots,5$, all indices modulo 5.
Now we prove that $u_Lu_5\notin E$. Suppose that $u_Lu_5\in E$. Now we discuss two cases, depending on $u_1$.

\case{1} $u_1\in B_2$.

Since $B_1$ and $B_2$ are $P_3$-free, $u_2$ and $u_5$ belong to $B_1$ and $B_2$, respectively, by symmetry assume that $u_2\in B_1$, $u_5\in B_2$. Note that for every $P_3$ in $L$, there is a vertex in $L$ which is anti-complete to that $P_3$. Let $v_i\in L$ be a vertex which is anti-complete to $\{u_2,u_1,u_5\}$, then $v_iu_L\notin E$ since $N_L(u_L)\subseteq N_L(u_1)$. Then $\{a_2,u_1,u_5,v_i\}$ or $\{a_1,u_L,u_2,v_i\}$ induces a $K_3+K_1$.

\case{2} $u_1\in B_1$.

Since $u_L$ and $u_1$ are not comparable in $G$, $N_{A\cup B_3}(u_L)\nsubseteq N_{A\cup B_3}(u_1)$. Let $b\in B_3$ such that $bu_L\in E$, $bu_1\notin E$. Since $B_1$ and $B_2$ are $P_3$-free, $u_2u_5\in B_2$, and at least one of $u_3$ and $u_4$ is in $B_1$. If $u_3,u_4\in B_1$, then one of $u_3$ and $u_4$ is adjacent to $b$, or else $\{a_1,u_3,u_4,b\}$ induces a $K_3+K_1$. If one of $u_3$ and $u_4$ is in $B_2$, say $u_3$, then one of $u_2$ and $u_3$ is adjacent to $b$, or else $\{b,a_3,a_1,u_1,u_2,u_3\}$ induces a $P_6$. So $b$ is adjacent to at least one of $u_2,u_3,u_4,u_5$. Suppose that $bu_2\in E$. Then each vertex in $V\setminus\{b,u_L,u_2\}$ is adjacent to exactly one vertex in $\{b,u_L,u_2\}$. So $bu_4\in E$, $bu_3,bu_5\notin E$. Suppose that $bu_3\in E$, then we have $bu_2\notin E$. 
Then $bu_4\notin E$, or else $\{b,u_3,u_4,u_1\}$ induces a $K_3+K_1$. Then $bu_5\in E$, or else $\{b,a_3,a_1,u_1,u_5,u_4\}$ induces a $P_6$. So by symmetry we may assume that $bu_2,bu_4\in E$, $bu_3,bu_5\notin E$. Note that for every $K_2$ in $L$, there is another $K_2$ which is anti-complete to the former $K_2$. Let $\{v_i,v_j\}$ be a $K_2$ in $L$ which is anti-complete to $\{u_1,u_2\}$. Since every vertex in $V\setminus\{u_2,u_L,b\}$ has exactly one neighbour in $\{u_2,u_L,b\}$, we have $v_ib,v_jb\in E$. Then $\{b,v_i,v_j,u_1\}$ induces a $K_3+K_1$.

So $u_Lu_5\notin E$. Note that every $P_3$ in $L$ is in a $C_5$ in $L$. Since $\{u_1,u_2,u_3,u_4,u_5\}$ is an arbitrary $C_5$ in $L$ containing $u_1,u_2$, $u_L$ has only one neighbour $u_2$ in $L$. Note that for every vertex $v$ in $L$, there is a $P_5$ in $L$ such that $v$ is an end of the $P_5$. Then $L\cup\{u_L\}$ contains a $P_6$. So $u_L$ does not exist. This implies that $B_1\cup B_2$ contains no comparable vertices. Then $B_1\cup B_2=L$ is an induced subgraph of the Clebsch graph.
\end{proof}

Since every $B\setminus B_i$ contains at most 16 vertices, $G$ contains at most 27 vertices. So there are only finitely many 6-vertex-critical ($P_6$,diamond)-free graphs with clique number 3.
\end{proof}

\end{document}